\documentclass[11pt]{article}%
\usepackage{amssymb}
\usepackage{amsfonts}
\usepackage{amsmath}
\usepackage{graphicx}%
\providecommand{\U}[1]{\protect\rule{.1in}{.1in}}
\newtheorem {theorem}{Theorem}[section]

\newtheorem {corollary}[theorem]{Corollary}

\newtheorem {definition}[theorem]{Definition}
\newtheorem {example}[theorem]{Example}

\newtheorem {lemma}[theorem]{Lemma}

\newtheorem {problem}[theorem]{Problem}
\newtheorem {proposition}[theorem]{Proposition}

\newenvironment {proof}[1][Proof]{\noindent \textbf {#1.} }{\ \rule {0.5em}{0.5em}}
\begin{document}
\date{}
\title{Multiplication and composition operators on the derivative Hardy space
$S^{2}({\mathbb{D}})$ }
\author{Caixing Gu and Shuaibing Luo}
\maketitle

\begin{abstract}
In this paper we propose a different (and equivalent) norm on $S^{2}%
({\mathbb{D}})$ which consists of functions whose derivatives are in the Hardy
space of unit disk. The reproducing kernel of $S^{2}({\mathbb{D}})$ in this
norm admits an explicit form, and it is a complete Nevanlinna-Pick kernel.
Furthermore, there is a surprising connection of this norm with $3$%
-isometries. We then study composition and multiplication operators on this
space. Specifically, we obtain an
upper bound for the norm of $C_{\varphi}$ for a class of composition
operators. We completely characterize multiplication operators which are
$m$-isometries. As an application of the 3-isometry, we describe the reducing
subspaces of $M_{\varphi}$ on $S^{2}({\mathbb{D}})$ when $\varphi$ is a finite Blaschke product of order 2.

\end{abstract}

\section{Introduction}

Let ${\mathbb{D}}$ be the open unit disk of the complex plane ${\mathbb{C}}$.
Let $H({\mathbb{D)}}$ be the space of all analytic functions on ${\mathbb{D}}%
$. If $\varphi$ is an analytic self-map of ${\mathbb{D}}$, then the
composition operator $C_{\varphi}$ with symbol $\varphi$ is defined on
$H({\mathbb{D)}}$ by $C_{\varphi}f=f\circ\varphi$ for $f\in H({\mathbb{D)}}$.
The multiplication operator $M_{\psi}$ with symbol $\psi\in H({\mathbb{D)}} $
is defined by $M_{\psi}f=\psi f$ for $f\in H({\mathbb{D)}}$. The power series
of $f\in H({\mathbb{D)}}$ is (always) denoted by
\[
f(z)={\sum_{n\geq0}}f_{n}z^{n},z\in{\mathbb{D}}%
\]
Multiplication operators and composition operators on subspaces of
$H({\mathbb{D)}}$ such as the Hardy space, the Bergman space and the Dirichlet
space have been studied extensively in literature \cite{AM} \cite{CM}
\cite{EM} \cite{GH}.

The definitions of the Hardy space $H^{2}({\mathbb{D}})$ and the Bergman space
$A^{2}({\mathbb{D}})$ are standard. There are two definitions of the Dirichlet
space \cite{EM}, in this paper we take the one connected with $2$-isometries.
\begin{equation}
D^{2}({\mathbb{D)}}{\mathbb{=}}\left\{  f\in H({\mathbb{D)}}:\left\Vert
f\right\Vert _{D^{2}}^{2}=\left\Vert f\right\Vert _{H^{2}}^{2}+\left\Vert
f^{\prime}\right\Vert _{A^{2}}^{2}={\sum_{n\geq0}}\left(  n+1\right)
\left\vert f_{n}\right\vert ^{2}<\infty\right\}  .\label{defdiri}%
\end{equation}
The reproducing kernels of $H^{2}({\mathbb{D)}}$, $A^{2}({\mathbb{D)}}$, and
$D^{2}({\mathbb{D)}}$ are, respectively,
\begin{equation}
K_{w}(z)=\frac{1}{1-\overline{w}z},K_{w}(z)=\frac{1}{\left(  1-\overline
{w}z\right)  ^{2}},K_{w}(z)=\frac{1}{\overline{w}z}\ln\frac{1}{1-\overline
{w}z}.\label{kernel}%
\end{equation}

In this paper we will study composition operators and multiplication operators
on the space of analytic functions whose derivatives are in the Hardy space.
It turns out there are three (norm equivalent) definitions of this space in literature.

The first definition, denoted by $S^{2}({\mathbb{D)}}$, is defined by
\[
S^{2}({\mathbb{D)}}{\mathbb{=}}\left\{  f\in H({\mathbb{D)}}:\left\Vert
f\right\Vert _{S^{2}}^{2}=\left\vert f(0)\right\vert ^{2}+\left\Vert
f^{\prime}\right\Vert _{H^{2}}^{2}=\left\vert f(0)\right\vert ^{2}%
+{\sum_{n\geq1}}n^{2}\left\vert f_{n}\right\vert ^{2}<\infty\right\}  ,
\]
see \cite{AHP}, \cite{Heller}, \cite{Heller2}, \cite{Barbara}. Composition
operators and linear isometries on Banach space $S^{p}({\mathbb{D)}}$ were
studied in 1978 \cite{Roan78} and 1985 \cite{Novinger85}.

The second definition, which will be denoted by $S_{2}^{2}({\mathbb{D)}}$, is
defined by
\[
S_{2}^{2}({\mathbb{D)}}{\mathbb{=}}\left\{  f\in H({\mathbb{D)}}:\left\Vert
f\right\Vert _{S_{2}^{2}}^{2}=\left\Vert f\right\Vert _{H^{2}}^{2}+\left\Vert
f^{\prime}\right\Vert _{H^{2}}^{2}={\sum_{n\geq0}}(1+n^{2})\left\vert
f_{n}\right\vert ^{2}<\infty\right\}  .
\]
Korenblum in 1972 \cite{Ko72} characterized the invariant subspaces of $M_{z}
$ on $S_{2}^{2}({\mathbb{D)}}$. A recent paper \cite{CuBh} describes the
lattice of invariant subspaces of the shift plus complex Volterra operator on
the Hardy space by using Korenblum's result.

The third definition, which will be denoted by $S_{3}^{2}({\mathbb{D)}}$,
\[
S_{3}^{2}({\mathbb{D)}}=\left\{  f\in H({\mathbb{D)}}:f={\sum_{n\geq0}}%
f_{n}z^{n},{\text{ }}\left\Vert f\right\Vert _{S_{3}^{2}}^{2}={\sum_{n\geq0}%
}(1+n)^{2}\left\vert f_{n}\right\vert ^{2}<\infty\right\}  .
\]
The space $S_{3}^{2}({\mathbb{D)}}$ is just one of the Dirichlet type spaces
$D_{\alpha}$ space with $\alpha=2,$ where
\[
D_{\alpha}=\left\{  f\in H({\mathbb{D)}}:f={\sum_{n\geq0}}f_{n}z^{n},{\text{
}}\left\Vert f\right\Vert _{\alpha}^{2}={\sum_{n\geq0}}(1+n)^{\alpha
}\left\vert f_{n}\right\vert ^{2}<\infty\right\}  ,
\]
see \cite{BS84} \cite{Sh}. For $\alpha\geq0,$ the space $D_{\alpha}$ has the
complete Pick property, see Corollary 7.41 in \cite{AM} where $D_{\alpha}$ is
denoted by $H_{s}$ with $s=-\alpha.$

It is the purpose of this paper to define an equivalent norm on $S^{2}%
({\mathbb{D)}}$, and denote this new space by $S_{1}^{2}({\mathbb{D)}}$.
\begin{align}
S_{1}^{2}({\mathbb{D)}} &  {\mathbb{=}}\left\{  f\in H({\mathbb{D)}%
}:\left\Vert f\right\Vert _{S_{1}^{2}}^{2}=\left\Vert f\right\Vert _{H^{2}%
}^{2}+\frac{3}{2}\left\Vert f^{\prime}\right\Vert _{A^{2}}^{2}+\frac{1}%
{2}\left\Vert f^{\prime}\right\Vert _{H^{2}}^{2}<\infty\right\} \label{defs1}%
\\
&  =\left\{  f\in H({\mathbb{D)}}:\left\Vert f\right\Vert _{S_{1}^{2}}%
^{2}={\sum_{n\geq0}}\frac{\left(  n+1\right)  (n+2)}{2}\left\vert
f_{n}\right\vert ^{2}<\infty\right\}  .\nonumber
\end{align}
It seems appropriate to call $S_{1}^{2}({\mathbb{D)}}$ the derivative Hardy
space. There are three motivations. First, the reproducing kernel of
$S_{1}^{2}({\mathbb{D)}}$ has an "explicit" form, see Lemma \ref{rkosotd}.
Second, the reproducing kernel of
$S_{1}^{2}({\mathbb{D)}}$ is a complete Nevalinna-Pick Kernel. Third, a
surprising connection with $3$-isometries was made. Namely we will prove that
$M_{z}$ is a $3$-isometry on $S_{1}^{2}({\mathbb{D)}}$. In comparison,
$S^{2}({\mathbb{D)}}$ does not have the complete Pick property, and $M_{z}$ on
$S^{2}({\mathbb{D)}}$ is not a $3$-isometry; $S_{2}^{2}({\mathbb{D)}}$ does
not have the complete Pick property, but $M_{z}$ on $S_{2}^{2}({\mathbb{D)}}$
is a $3$-isometry; the reproducing kernel of $S_{3}^{2}({\mathbb{D)}}$ is a
complete Nevalinna-Pick Kernel, and $M_{z}$ on $S_{3}^{2}({\mathbb{D)}}$ is
also a $3$-isometry, but the reproducing kernel of $S_{3}^{2}({\mathbb{D)}}$
does not have an "explicit" form. In addition, one can define an equivalent
norm on the $D_{\alpha}$ space $(\alpha\in{\mathbb{N}}$) analogously so that
$M_{z}$ is an $m$-isometry ($m=\alpha+1$) on it. As an application of the
3-isometry $M_{z}$ on $S_{1}^{2}({\mathbb{D}})$, we describe the reducibility
of $M_{\psi}$ on $S^{2}({\mathbb{D}})$, where $\psi$ is a finite Blaschke
product of order 2.

Questions such as when multiplication operators and composition operators are
bounded or compact have the same answers with respect to equivalent norms.
Questions such as when multiplication operators and composition operators are
isometries or normal operators have different answers with respect to
equivalent norms. In Section 2, we make a few remarks about multiplication operators
on $S_{1}^{2}({\mathbb{D)}}$. In particular we present some sharp inequalities
for multiplication operators. As a consequence, we remark that $S_{1}%
^{2}({\mathbb{D)}}$ is an algebra which is known in literature. In Section 3,
we completely characterize when the multiplication operator $M_{\psi}$ on
$S_{1}^{2}({\mathbb{D)}}$ is an $m$-isometry for some $m\geq1.$ In Section 4,
by exploiting the properties of the 3-siometry $M_{\psi}$ on $S_{1}%
^{2}({\mathbb{D}})$ and the norm relation between $S_{1}^{2}({\mathbb{D}})$
and $S^{2}({\mathbb{D}})$, we obtain the reducibility of $M_{\psi}$ on
$S^{2}({\mathbb{D}})$ when $\psi$ is a finite Blaschke product of order 2. In
Section 5, we extend the norm estimates for composition operators on the
Hardy space and the Bergman space by Jury \cite{Jury} to $S_{1}^{2}%
({\mathbb{D)}}$. In the last section, we briefly remark that it is possible to
extend the results in this paper to the space of analytic functions whose
$n$-th derivatives are in the Hardy space or the Bergman space, those spaces
are equivalent to $D_{\alpha}$ with $\alpha$ being a positive integer.

\begin{lemma}
\label{rkosotd} The reproducing kernel of $S_{1}^{2}({\mathbb{D)}}$ is
\[
K_{w}(z)=\frac{2}{(\overline{w}z)^{2}}\left[ \overline{w}z+\left( \overline
{w}z-1\right) \ln\frac{1}{1-\overline{w}z}\right] .
\]

\end{lemma}

\begin{proof}
Note that for $\overline{w}z\neq0,$
\begin{align*}
K_{w}(z) &  ={\sum_{n\geq0}}\frac{2}{\left(  n+1\right)  (n+2)}(\overline
{w}z)^{n}={\sum_{n\geq0}}\left[  \frac{2}{\left(  n+1\right)  }-\frac
{2}{\left(  n+2\right)  }\right]  (\overline{w}z)^{n}\\
&  ={\sum_{n\geq0}}\frac{2}{\left(  n+1\right)  }(\overline{w}z)^{n}%
-{\sum_{n\geq0}}\frac{2}{\left(  n+2\right)  }(\overline{w}z)^{n}\\
&  =\frac{2}{\overline{w}z}\ln\frac{1}{1-\overline{w}z}-\frac{2}{\left(
\overline{w}z\right)  ^{2}}\left[  \ln\frac{1}{1-\overline{w}z}-\overline
{w}z\right] \\
&  =\frac{2}{(\overline{w}z)^{2}}\left[  \overline{w}z+\left(  \overline
{w}z-1\right)  \ln\frac{1}{1-\overline{w}z}\right]  ,
\end{align*}
and $K_{w}(z)=1$ when $\overline{w}z=0.$
\end{proof}

We remark here that by using Lemma \ref{rkosotd}, one can derive explicit formulas for the
adjoints of linear fractional composition operators $C_{\varphi}.$ Since the formulas for $C_{\varphi}^{\ast}$ are complicated and we haven't found the applications of them, we don't discuss them here.

\section{Some remarks on multiplication operators}

The fact that $D_{\alpha}$ is an algebra for $\alpha>1$ was proved
\cite{Kopp69} in 1969. The fact $S_{2}^{2}({\mathbb{D)}}$ and the analogous
$D_{n} $ for an integer $n\geq1$ are algebras were also mentioned in
\cite{Ko72} by referring to literature in Russian. The recent paper \cite{AHP}
proved that $S^{2}({\mathbb{D)}}$ is an algebra, and another recent paper
\cite{CuBh} also proved $S_{2}^{2}({\mathbb{D)}}$ is an algebra. As a
consequence, as mentioned both in \cite{Sh} and \cite{Ko72}, the maximal ideal
space of $D_{\alpha}$ for $\alpha>1$ (more precisely, the norm-equivalent
Banach algebra) is the closed unit disk. Therefore $M_{\psi} $ is bounded on
$D_{\alpha}$ if and only if $\psi\in D_{\alpha}$ and the spectrum of $M_{\psi
}$ is
\begin{equation}
\sigma(M_{\psi})=\psi(\overline{{\mathbb{D}}}).\label{spectrum}%
\end{equation}
This result in the special case $S^{2}({\mathbb{D)}}$ is also proved as
Theorem 4.3 in \cite{AHP}.

Here we will give another proof that $S_{1}^{2}({\mathbb{D)}}$ is an algebra,
since the proof is short, the constant in our inequality is sharp, and we will
use this constant in section 5.

We first note that $S_{1}^{2}({\mathbb{D}})$ contains functions in
$H({\mathbb{D}})$ whose power series have convergence radius strictly bigger
than $1.$ Since if ${\sum_{n\geq0}}\left\vert f_{n}\right\vert r^{n}<\infty$
for some $r>1,$ then $\left\vert f_{n}\right\vert \leq C(1/r^{n})$ and
\[
\left\Vert f\right\Vert _{S_{1}^{2}}^{2}={\sum_{n\geq0}}\frac{(n+1)(n+2)}%
{2}\left\vert f_{n}\right\vert ^{2}\leq C^{2}{\sum_{n\geq0}}\frac
{(n+1)(n+2)}{2r^{2n}}<\infty.
\]
Therefore $S_{1}^{2}({\mathbb{D}})$ contains functions which are analytic in a
neighborhood of the closed disk $\overline{{\mathbb{D}}},$ in particular all
rational functions whose poles are outside $\overline{{\mathbb{D}}}.$

The next proposition is similar to Proposition 2.2 in \cite{AHP}, but here our
estimates are tight which leads to a sharp constant.

\begin{proposition}
\label{sharp}Let $f\in S_{1}^{2}({\mathbb{D}}).$ Then $\left\Vert f\right\Vert
_{\infty}\leq\sqrt{2}\left\Vert f\right\Vert _{S_{1}^{2}}. $ Furthermore the
constant $\sqrt{2}$ is sharp.
\end{proposition}

\begin{proof}
Write $f(z)={\sum_{n\geq0}}f_{n}z^{n}.$ Then by Cauchy-Schwarz inequality,
\begin{align}
\left\vert f(z)\right\vert  &  \leq{\sum_{n\geq0}}\left\vert f_{n}\right\vert
\left\vert z\right\vert ^{n}={\sum_{n\geq0}}\sqrt{\frac{(n+1)(n+2)}{2}%
}\left\vert f_{n}\right\vert \sqrt{\frac{2}{(n+1)(n+2}}\left\vert z\right\vert
^{n}\nonumber\\
&  \leq\left(  {\sum_{n\geq0}}\frac{(n+1)(n+2)}{2}\left\vert f_{n}\right\vert
^{2}\right)  ^{1/2}\left(  {\sum_{n\geq0}}\frac{2}{(n+1)(n+2)}\left\vert
z\right\vert ^{2n}\right)  ^{1/2}\label{equal}\\
&  =\left\Vert f\right\Vert _{S_{1}^{2}}K_{z}(z)^{1/2}\leq\sqrt{2}\left\Vert
f\right\Vert _{S_{1}^{2}}.\nonumber
\end{align}
The inequality in (\ref{equal}) becomes an equality (at $z=1$) if $\left\vert
f_{n}\right\vert =\frac{2}{(n+1)(n+2)}.$ Indeed, if
\begin{equation}
f(z)={\sum_{n\geq0}}\frac{2}{(n+1)(n+2)}z^{n},\label{extremal}%
\end{equation}
then $\left\Vert f\right\Vert _{\infty}=f(1)=2$ and $\left\Vert f\right\Vert
_{S_{1}^{2}}=\sqrt{2},$ which shows the constant $\sqrt{2}$ is sharp.
\end{proof}

Another short proof will be
\[
\left\vert f(w)\right\vert =\left\vert \left\langle f,K_{w}\right\rangle
\right\vert \leq\left\Vert f\right\Vert _{S_{1}^{2}}\left\Vert K_{w}%
\right\Vert _{S_{1}^{2}}\leq\left\Vert f\right\Vert _{S_{1}^{2}}K_{w}%
(w)^{1/2}\leq\sqrt{2}\left\Vert f\right\Vert _{S_{1}^{2}}.
\]
But the argument in the lemma also shows that for $\left\vert z\right\vert
\leq1,$
\begin{align*}
\left\vert f(z)-{\sum_{n=0}^{N}}f_{n}z^{n}\right\vert  &  \leq\left(
{\sum_{n\geq N+1}}\frac{(n+1)(n+2)}{2}\left\vert f_{n}\right\vert ^{2}\right)
^{1/2}\left(  {\sum_{n\geq N+1}}\frac{2}{(n+1)(n+2)}\left\vert z\right\vert
^{2n}\right)  ^{1/2}\\
&  \leq\sqrt{\frac{2}{N+2}}\left\Vert f\right\Vert _{S_{1}^{2}}.
\end{align*}
Thus $f(z)$ is the uniform limit of its partial sum, so $f(z)$ extends
continuous to the closed disk $\overline{{\mathbb{D}}}$ and $S_{1}%
^{2}({\mathbb{D}})$ is contained in the disk algebra $A({\mathbb{D}}).$
Moreover, let $A^{1}({\mathbb{D}})=\left\{  f\in A({\mathbb{D}}):f^{\prime
}(z)\in A({\mathbb{D}})\right\}  .$ Then
\[
A^{1}({\mathbb{D}})\subseteq S_{1}^{2}({\mathbb{D}})\subseteq A({\mathbb{D}}).
\]
It turns out $S_{1}^{2}({\mathbb{D}})$ is an algebra.

\begin{theorem}
\label{upperbound}The space $S_{1}^{2}({\mathbb{D}})$ is an algebra.
Furthermore if $f,g\in S_{1}^{2}({\mathbb{D}}),$ then
\[
\left\Vert fg\right\Vert _{S_{1}^{2}}\leq2\sqrt{2}\left\Vert f\right\Vert
_{S_{1}^{2}}\left\Vert g\right\Vert _{S_{1}^{2}},
\]
and the equality holds only when $f=0$ or $g=0.$
\end{theorem}

\begin{proof}
Let $f,g\in S_{1}^{2}({\mathbb{D}}).$ Then
\begin{align*}
\left\Vert fg\right\Vert _{S_{1}^{2}}^{2} & =\left\Vert fg\right\Vert _{H^{2}%
}^{2}+\frac{3}{2}\left\Vert f^{\prime}g+fg^{\prime}\right\Vert _{A^{2}}%
^{2}+\frac{1}{2}\left\Vert f^{\prime}g+fg^{\prime}\right\Vert _{H^{2}}^{2}\\
& \leq\left\Vert f\right\Vert _{\infty}^{2}\left\Vert g\right\Vert _{H^{2}%
}^{2}+\|f\|_{H^{2}}^{2}\|g\|_{\infty}^{2}+\frac{3}{2}\left( 2\|f\|_{A^{2}}%
^{2}\|g\|_{\infty}^{2}+2\|f\|_{\infty}^{2}\|g\|_{A^{2}}^{2}\right) \\
& \hspace{0.5cm}+\frac{1}{2}\left( 2\|f\|_{H^{2}}^{2}\|g\|_{\infty}%
^{2}+2\|f\|_{\infty}^{2}\|g\|_{H^{2}}^{2}\right) \\
& \leq2\|f\|_{S_{1}^{2}}^{2}\|g\|_{\infty}^{2}+2\|f\|_{\infty}^{2}%
\|g\|_{S_{1}^{2}}^{2}\\
& \leq8\|f\|_{S_{1}^{2}}^{2}\|g\|_{S_{1}^{2}}^{2},
\end{align*}
where we used Proposition \ref{sharp} in the last inequality. Moreover, the
equality holds only when $f=0$ or $g=0$.
\end{proof}

It is well-known that the reproducing kernel $K_{w}(z)$ is an eigenvector of
$M_{f}^{\ast}$ since $M_{f}^{\ast}K_{w}(z)=\overline{f(w)}K_{w}(z).$ Thus
$\left\Vert M_{f}\right\Vert =\left\Vert M_{f}^{\ast}\right\Vert
\geq\left\Vert f\right\Vert _{\infty}.$ We have the following result which is
similar to Theorem 3.3 in \cite{AHP}.

\begin{theorem}
\label{norm}Let $f\in H({\mathbb{D}}).$ Then $M_{f}$ is a bounded linear
operator on $S_{1}^{2}({\mathbb{D}})$ if and only if $f\in S_{1}%
^{2}({\mathbb{D}}).$ Furthermore,
\begin{equation}
\max\left\{ \left\Vert f\right\Vert _{\infty},\left\Vert f\right\Vert
_{S_{1}^{2}}\right\} \leq\left\Vert M_{f}\right\Vert \leq2\sqrt{2}\left\Vert
f\right\Vert _{S_{1}^{2}}.\label{banacgalgebra}%
\end{equation}

\end{theorem}

\begin{proof}
The upper bound of $\left\Vert M_{f}\right\Vert $ is by Theorem
\ref{upperbound}. The lower bound of $\left\Vert M_{f}\right\Vert $ follows by
noting that $\left\Vert M_{f}\right\Vert \geq\left\Vert M_{f}1\right\Vert
_{S_{1}^{2}}=\left\Vert f\right\Vert _{S_{1}^{2}}.$
\end{proof}

Recall that the multiplier (space) of a holomorphic function space $H$ is the set
of $f$ in $H$ such that $M_{f}$ is a bounded operator on $H.$ The multiplier
of the Hardy space $H^{2}({\mathbb{D)}}$ or the Bergman space $A^{2}%
({\mathbb{D)}}$ is $H^{\infty}({\mathbb{D)}}$. The multiplier of Dirichlet
space $D^{2}({\mathbb{D)}}$ is strictly contained in $H^{\infty}({\mathbb{D)}%
}$ \cite{EM}. It is interesting to note that the multiplier of $S_{1}%
^{2}({\mathbb{D}})$ is $S_{1}^{2}({\mathbb{D}})$ itself as a set. But they
carry different norms. In particular, $S_{1}^{2}({\mathbb{D}})$ is an algebra,
but not a Banach algebra, see Example \ref{notequal} below.

It is interesting to know when the equalities holds in (\ref{banacgalgebra}).
Here we give some examples, and we will see in the last section that if $f $
is not a constant, then $\|f\|_{\infty}<\|M_{f}\|$.

\begin{example}
For each natural number $k,$ $\left\Vert M_{z^{k}}\right\Vert =\left\Vert
z^{k}\right\Vert _{S_{1}^{2}}=\sqrt{\frac{(k+1)(k+2)}{2}}.$ Note that
\[
\frac{(n+k+1)(n+k+2)}{2}\leq\frac{(k+1)(k+2)}{2}\frac{(n+1)(n+2)}{2},n\geq0.
\]
Therefore
\[
\left\Vert z^{k}f\right\Vert _{S_{1}^{2}}^{2}={\sum_{n\geq0}}\frac
{(n+k+1)(n+k+2)}{2}\left\vert f_{n}\right\vert ^{2}\leq\left\Vert
z^{k}\right\Vert _{S_{1}^{2}}^{2}\left\Vert f\right\Vert _{S_{1}^{2}}^{2}.
\]
This completes the proof.
\end{example}

\begin{example}
\label{notequal}$\left\Vert M_{1+z}\right\Vert >\sqrt{4.5}>\left\Vert
1+z\right\Vert _{S_{1}^{2}}=2.$ Let $f(z)$ be as in (\ref{extremal}),
\[
f(z)={\sum_{n\geq0}}\frac{2}{(n+1)(n+2)}z^{n},
\]
Then $\left\Vert f\right\Vert _{S_{1}^{2}}=\sqrt{2}.$ A direct computation
shows that
\begin{align*}
(1+z)f(z) &  =1+{\sum_{n\geq1}}\frac{4}{n(n+2)}z^{n}\\
\left\Vert (1+z)f\right\Vert _{S_{1}^{2}}^{2} &  =1+{\sum_{n\geq1}}%
\frac{(n+1)(n+2)}{2}\left[  \frac{4}{n(n+2)}\right]  ^{2}\\
&  =1+8{\sum_{n\geq1}}\frac{(n+1)}{n^{2}(n+2)}>1+8{\sum_{n\geq1}}\frac
{1}{n(n+1)}=9.
\end{align*}
Therefore $\left\Vert M_{1+z}\right\Vert >\sqrt{4.5}.$ Since
\[
\left\Vert (1+z)f\right\Vert _{S_{1}^{2}}>\sqrt{9}>\sqrt{8}=\left\Vert
1+z\right\Vert _{S_{1}^{2}}\left\Vert f\right\Vert _{S_{1}^{2}},
\]
this example also demonstrates that $S_{1}^{2}({\mathbb{D}})$ is not a Banach algebra.
\end{example}

\section{Finite Blaschke products and $m$-isometries}

It is known that the only inner functions belong to Dirichlet space
$D^{2}({\mathbb{D}})$ are finite Blaschke products, so this is also true for
the space $S_{1}^{2}({\mathbb{D}}).$ It has been observed in \cite{AHP} that
the only isometric multiplication operators on $S^{2}({\mathbb{D}})$ are the
constant multiples of the identity operator. Here we prove a stronger result
that $M_{f}$ on $S_{1}^{2}({\mathbb{D}})$ is a $2$-isometry if and only if it
is a constant multiple of the identity operator. But we first show that
$M_{z}$ on $S_{1}^{2}({\mathbb{D}})$ is a $3$-isometry, which explains why
there are no nontrivial isometric or $2$-isometric multiplication operators on
$S^{2}({\mathbb{D}}).$

A systematic study of $m$-isometries was initiated by Agler and Stankus in a
series of three papers, the first one is \cite{AS}. The theory for
$m$-isometries on Hilbert spaces has rich connections to Toeplitz operators,
classical function theory, ordinary differential equations and other areas of
mathematics. The work of Richter \cite{Richter} and \cite{R} on analytic
$2$-isometries has a connection with the invariant subspaces of the shift
operator on the Dirichlet space. Recently complete characterizations of
$m$-isometric weighted shifts were obtained \cite{BMN} \cite{Guieot}, from
which it follows easily that $M_{z}$ on $S_{1}^{2}({\mathbb{D}})$ is a
$3$-isometry. However, finding the connection between a natural function space
such as $S_{1}^{2}({\mathbb{D}})$ and $3$-isometries is exciting, since\ it
may lead to better understanding of both general $3$-isometries and the space
$S_{1}^{2}({\mathbb{D}})$ as in the case of $2$-isometries and Dirichlet
space. There are also study of $(m,p)$-isometries and related operators on
Banach spaces recently, see \cite{Gujmaa} \cite{Guszged} and references therein.

The operator $T$ on $H$ is an $m$-isometry for some positive integer $m$ as in
\cite{AS}, if
\begin{align*}
\beta_{m}(T) & :=(\overline{z}z-1)^{m}(T)=\sum\limits_{k=0}^{m}(-1)^{m-k}%
\binom{m}{k}\overline{z}^{k}z^{k}|_{\overline{z}=T^{\ast},z=T}\\
& =\sum\limits_{k=0}^{m}(-1)^{m-k}\binom{m}{k}T^{\ast k}T^{k}=0,
\end{align*}
where $T^{\ast}$ is always on the left of $T$. Equivalently
\[
\left\langle \beta_{m}(T)h,h\right\rangle =\sum\limits_{k=0}^{m}%
(-1)^{m-k}\binom{m}{k}\left\Vert T^{k}h\right\Vert ^{2}=0{\text{ for all }%
}h\in H.
\]
We say $T$ is a strict $m$-isometry if $T$ is an $m$-isometry but not an
$(m-1)$-isometry.

Let $\left\{ e_{n},n\geq0\right\} $ be the standard bases of $l_{2}.$ The
operator $T$ is a unilateral weighted shift with weight sequence $\left\{
w_{n}\right\} $ if $Te_{n}=w_{n}e_{n+1},n\geq0.$ The following result is from
Corollary 4.6 \cite{Guieot}, where more general results for both unilateral
weighted shifts and bilateral weighted shifts on $l_{p}$ space are obtained.

\begin{lemma}
\label{polycoro}Let $T$ be a unilateral shift with weights $w_{n}$ on $l_{2}.
$ Then $T$ is a strict $m$-isometry if and only if there exists a polynomial
$P(x)$ of degree $m-1$ such that $P(n)>0$ for $n\geq0$ and
\begin{equation}
\left\vert w_{n}\right\vert ^{2}=\frac{P(n+1)}{P(n)}{\text{ for }}%
n\geq0.\label{wnp}%
\end{equation}

\end{lemma}

\begin{lemma}
The multiplication operator $M_{z}$ on $S_{1}^{2}({\mathbb{D}})$ is a strict
$3$-isometry.
\end{lemma}

\begin{proof}
Note that $M_{z}z^{n}=z^{n+1}.$ Thus
\[
M_{z}\frac{z^{n}}{\left\Vert z^{n}\right\Vert }=\frac{\left\Vert
z^{n+1}\right\Vert }{\left\Vert z^{n}\right\Vert }\frac{z^{n+1}}{\left\Vert
z^{n+1}\right\Vert }=\sqrt{\frac{(n+3)}{(n+1)}}\frac{z^{n+1}}{\left\Vert
z^{n+1}\right\Vert }.
\]
Hence $M_{z}$ is a weighted shift with weights $w_{n}=\sqrt{(n+3)/(n+1)}.$
Note that
\[
w_{n}^{2}=\frac{P(n+1)}{P(n)}{\text{ for }}P(n)=(n+1)(n+2),n\geq0.
\]
By previous lemma, $M_{z}$ is a strict $3$-isometry.
\end{proof}

Since $M_{z}$ on $S^{2}({\mathbb{D}})$ is also a weighted shift, a similar
argument shows that $M_{z}$ on $S^{2}({\mathbb{D}})$ is not an $m$-isometry
for any $m\geq1.$ Since the spectrum of weighted shifts has been studied
extensively in the past \cite{Sh}, we have that $\sigma(M_{z})=\overline
{{\mathbb{D}}},$ $\sigma_{e}(M_{z})={\mathbb{T}}$. Furthermore, for each
$\lambda\in{\mathbb{D}}$, $M_{z}-\lambda$ is left invertible, and $\ker
(M_{z}^{\ast}-\overline{\lambda}I)$ is one dimensional, in fact, $\ker
(M_{z}^{\ast}-\overline{\lambda}I)$ is spanned by $K_{\lambda}.$ Therefore, if
$\theta$ is a finite Blaschke product, then $\theta(M_{z})$ is well-defined.
It is easy to see that $\theta(M_{z})$ is the multiplication operator
$M_{\theta}.$ The following result is Proposition 2.7 \cite{Guszged}, see
Theorem 2.10 \cite{Guszged} for a more general result.

\begin{lemma}
\label{finite}If $T$ is a strict $m$-isometry on $H$ and $\theta$ is a finite
Blaschke product, then $\theta(T)$ is a strict $m$-isometry on $H.$
\end{lemma}

The following theorem follows directly from the above two lemmas.

\begin{theorem}
\label{polyg}Let $\theta$ be a finite Blaschke product. Then $M_{\theta}$ on
$S_{1}^{2}({\mathbb{D)}}$ is a $3$-isometry. That is, for any $f\in S_{1}%
^{2}({\mathbb{D)}}$,
\[
\left\Vert \theta^{3}f\right\Vert ^{2}-3\left\Vert \theta^{2}f\right\Vert
^{2}+3\left\Vert \theta f\right\Vert ^{2}-\left\Vert f\right\Vert ^{2}=0,
\]
where the norm is given by (\ref{defs1}).
\end{theorem}

It is natural to ask if there are other $M_{\theta}$ which are $m$-isometries.
It turns out those identified in the above theorem are the only ones.

\begin{theorem}
\label{lessone}Assume $\psi\in S_{1}^{2}({\mathbb{D)}}$ is not a constant. If
$M_{\psi}$ is an $m$-isometry on $S_{1}^{2}({\mathbb{D}})$ for some $m\geq1$,
then $\psi$ is a finite Blaschke product. Consequently, $M_{\psi}$ is a strict
3-isometry on $S_{1}^{2}({\mathbb{D}})$.
\end{theorem}

\begin{proof}
By Lemma 1.21 \cite{AS}, the spectrum of an $m$-isometry is either equal to
$\overline{{\mathbb{D}}}$ or is contained in ${\mathbb{T}}$. If $M_{\psi}$ is
an $m$-isometry, by (\ref{spectrum}), $\sigma(M_{\psi})=\psi(\overline
{{\mathbb{D}}})=\overline{{\mathbb{D}}}.$ Also by page 385 in \cite{AS}, $\forall \lambda \in {\mathbb{D}}$, $M_\psi - \lambda$ is left semi-Fredholm, and $\text{ind} (M_\psi - \lambda) = \text{ind} M_\psi$. Since
\begin{align*}
\text{ind} (M_\psi - \lambda) &= \dim \ker (M_\psi - \lambda) - \dim \ker (M_\psi - \lambda)^* \\
&= 0 - \dim [S_{1}^{2}({\mathbb{D)}} \ominus (M_\psi - \lambda)S_{1}^{2}({\mathbb{D)}}],
\end{align*}
$\sigma_e(M_\psi) \supseteq {\mathbb{D}}$ or  $\sigma_e(M_\psi) \subseteq {\mathbb{T}}$. Note that $\sigma_e(M_\psi) = \psi({\mathbb{T}})$. Thus if $\sigma_e(M_\psi) \supseteq {\mathbb{D}}$, then $\psi({\mathbb{T}}) = \overline{{\mathbb{D}}}$, which contradicts the fact that $\psi' \in H^2({\mathbb{D}})$. So $\sigma_e(M_\psi) = \psi({\mathbb{T}}) \subseteq {\mathbb{T}}$. By maximum modulus principle, we also have $\psi({\mathbb{T}}) \supseteq {\mathbb{T}}$, therefore $\psi({\mathbb{T}}) = {\mathbb{T}}$. So $\psi$ is a finite Blaschke product. It then follows from Theorem \ref{polyg} that $M_{\psi}$ is a strict
3-isometry on $S_{1}^{2}({\mathbb{D}})$.
\end{proof}

Now we study the $m$-isometry on $S^{2}({\mathbb{D}})$. Recall
\[
S^{2}({\mathbb{D)}}{\mathbb{=}}\left\{  f\in H({\mathbb{D)}}:\left\Vert
f\right\Vert _{S^{2}}^{2}=\left\vert f(0)\right\vert ^{2}+\left\Vert
f^{\prime}\right\Vert _{H^{2}}^{2}=\left\vert f(0)\right\vert ^{2}%
+{\sum_{n\geq1}}n^{2}\left\vert f_{n}\right\vert ^{2}<\infty\right\}  .
\]
For notational simplicity, in the rest of this section and the next section we
write $\left\Vert f\right\Vert =\left\Vert f\right\Vert _{S_{1}^{2}%
},\left\Vert f\right\Vert _{1}=\left\Vert f\right\Vert _{S^{2}}$, $D(f)=\Vert
f^{\prime}\Vert_{A^{2}}^{2}$. Then
\begin{equation}
2\left\Vert f\right\Vert ^{2}=2\Vert f\Vert_{H^{2}}^{2}+3\Vert f^{\prime}%
\Vert_{A^{2}}^{2}+\Vert f^{\prime}\Vert_{H^{2}}^{2}=\Vert f\Vert_{1}%
^{2}+2\Vert f\Vert_{H^{2}}^{2}+3D(f)-|f(0)|^{2}\label{twonorms}%
\end{equation}
Before we study the $m$-isometry on $S^{2}({\mathbb{D}})$, We need some
lemmas. We first recall a formula, which is essentially (1.3) in \cite{AS};
see also Theorem 2.7 \cite{Guieot} for several more general formulas on Banach spaces.

\begin{lemma}
\label{polygrowth}Assume $T$ is an $m$-isometry. Then
\[
\left\Vert T^{n}h\right\Vert ^{2}=\sum\limits_{k=0}^{m-1}\binom{n}%
{k}\left\langle \beta_{k}(T)h,h\right\rangle ,{\text{ }}h\in H.
\]

\end{lemma}

We also need the following lemma. Recall that
\[
D^{2}({\mathbb{D}})=\left\{ f\in H({\mathbb{D)}}:\left\Vert f\right\Vert
_{D^{2}}^{2}=\left\Vert f\right\Vert _{H^{2}}^{2}+\left\Vert f^{\prime
}\right\Vert _{A^{2}}^{2}=\|f\|_{H^{2}}^{2}+D(f)\right\}
\]
is the Dirichlet space.

\begin{lemma}
\label{ffldirs} Let $\psi$ be a finite Blaschke product. Then for $f\in
D^{2}({\mathbb{D}}),n\geq0$,
\[
D(\psi^{n}f)=D(f)+n[D(\psi f)-D(f)].
\]

\end{lemma}

\begin{proof}
By \cite[Theorem 4.2]{RS91}, we have $M_{\psi}$ is a 2-isometry on
$D^{2}({\mathbb{D}})$, i.e.
\[
\|\psi^{2}f\|_{D_{2}}^{2}=2\|\psi f\|_{D_{2}}^{2}-\|f\|_{D_{2}}^{2},f\in
D_{2}({\mathbb{D}}).
\]
Since $M_{\psi}$ is an isometry on $H^{2}({\mathbb{D}})$, we obtain
\[
D(\psi^{2}f)=2D(\psi f)-D(f).
\]
Thus $M_{\psi}$ is a 2-isometry on the Hilbert space $D_{0}=\{f\in
D^{2}({\mathbb{D}}):f(0)=0,\|f\|_{*}^{2}:=D(f)\}$. It then follows from Lemma
\ref{polygrowth} that
\[
D(\psi^{n}f)=D(f)+n[D(\psi f)-D(f)],n\geq0.
\]
\end{proof}

\begin{proposition}
\label{igtisdn} Let $\psi$ be a finite Blaschke product. Then for $f\in
S^{2}({\mathbb{D}}),n\geq2$,
\begin{align*}
\|\psi^{n}f\|_{1}^{2} & =\frac{n(n-1)}{2}\|\psi^{2}f\|_{1}^{2}-n(n-2)\|\psi
f\|_{1}^{2}+\frac{(n-1)(n-2)}{2}\|f\|_{1}^{2}\\
& \hspace{0.5cm}-\frac{(n-1)(n-2)}{2}|f(0)|^{2}-\frac{n(n-1)}{2}|\psi
^{2}(0)f(0)|^{2}\\
& \hspace{0.5cm}+n(n-2)|\psi(0)f(0)|^{2}+|\psi^{n}(0)f(0)|^{2}.
\end{align*}

\end{proposition}

\begin{proof}
By Theorem \ref{polyg}, $M_{\psi}$ is a $3$-isometry on $S_{1}^{2}%
({\mathbb{D}}).$ Then by Lemma \ref{polygrowth}, we obtain for $n\geq2$,
\begin{align*}
\Vert\psi^{n}f\Vert^{2} & =\Vert f\Vert^{2}+n\left( \Vert\psi f\Vert^{2}-\Vert
f\Vert^{2}\right) +\frac{n(n-1)}{2}\left( \Vert\psi^{2}f\Vert^{2}%
-2\Vert\psi f\Vert^{2}+\Vert f\Vert^{2}\right) \\
& =\frac{n(n-1)}{2}\Vert\psi^{2}f\Vert^{2}-n(n-2)\Vert\psi f\Vert^{2}%
+\frac{(n-1)(n-2)}{2}\Vert f\Vert^{2}.
\end{align*}
Since $M_{\psi}$ is an isometry on $H^{2}({\mathbb{D}})$, by (\ref{twonorms})
and Lemma \ref{ffldirs}, we get
\begin{align*}
& \Vert\psi^{n}f\Vert_{1}^{2}-|\psi^{n}(0)f(0)|^{2}\\
& =\frac{n(n-1)}{2}\left( \Vert\psi^{2}f\Vert_{1}^{2}-|\psi^{2}(0)f(0)|^{2}%
\right) \\
& -n(n-2)\left( \Vert\psi f\Vert_{1}^{2}-|\psi(0)f(0)|^{2}\right)
+\frac{(n-1)(n-2)}{2}\left( \Vert f\Vert_{1}^{2}-|f(0)|^{2}\right)
\end{align*}
The conclusion then follows from the above equation.
\end{proof}

We remark that $M_{z}$ is also a 3-isometry on $S_{2}^{2}({\mathbb{D}})$, and
\begin{align}
\label{imisttgs}\|f\|_{S_{2}^{2}({\mathbb{D}})}^{2}=\|f\|_{S^{2}({\mathbb{D}%
})}^{2}+\|f\|_{H^{2}}^{2}-|f(0)|^{2}.
\end{align}
Thus we can also obtain Proposition \ref{igtisdn} from (\ref{imisttgs}).

Now we can show that $M_{\psi}$ is an $m$-isometry on $S^{2}({\mathbb{D}})$ if
and only if $\psi$ is a constant. The case for $m=1$ was proved in \cite{AHP}.
Notice that Lemma \ref{polygrowth} shows that if $T$ is a strict $m$-isometry,
then for some $h$, $\langle\beta_{m-1}(T)h,h\rangle\neq0$, so $\Vert
T^{n}h\Vert^{2}$ grows like a polynomial in $n$ of degree $m-1$ as
$n\rightarrow\infty$.

\begin{theorem}
Assume $\psi\in S^{2}({\mathbb{D}})$ is not a constant. Then $M_{\psi}$ can
not be an $m$-isometry on $S^{2}({\mathbb{D}})$ for any $m\geq1$.
\end{theorem}

\begin{proof}
If $M_{\psi}$ on $S^{2}({\mathbb{D}})$ is an $m$-isometry, the same argument
as in Theorem \ref{lessone} shows $\psi$ is a finite Blaschke product. Note
that for $h\in S^{2}({\mathbb{D}})$,
\begin{align*}
\Vert\psi^{n}h\Vert_{1}^{2} & =|\psi^{n}(0)h(0)|^{2}+\Vert(\psi^{n}h)^{\prime
}\Vert_{H^{2}}^{2}\\
& \leq\Vert h\Vert_{\infty}^{2}+\left( n\Vert\psi^{\prime}h\Vert_{H^{2}}+\Vert
h^{\prime}\Vert_{H^{2}}\right) ^{2}.
\end{align*}
That is, $\Vert M_{\psi}^{n}h\Vert_{1}^{2}$ grows at most as a polynomial in
$n$ of degree $2$. Therefore $M_{\psi}$ can not be a strict $m$-isometry on
$S^{2}({\mathbb{D}})$ for $m\geq4$.

If $M_{\psi}$ is a $3$-isometry on $S^{2}({\mathbb{D}})$, then for $h\in
S^{2}({\mathbb{D}})$,
\[
\Vert\psi^{3}h\Vert_{1}^{2}-3\Vert\psi^{2}h\Vert_{1}^{2}+3\Vert\psi h\Vert
_{1}^{2}-\Vert h\Vert_{1}^{2}=0.
\]
Hence by Proposition \ref{igtisdn}, we obtain
\[
-|\psi^{3}(0)h(0)|^{2}+3|\psi^{2}(0)h(0)|^{2}-3|\psi(0)h(0)|^{2}+|h(0)|^{2}=0.
\]
So $-|\psi^{3}(0)|^{2}+3|\psi^{2}(0)|^{2}-3|\psi(0)|^{2}+1=0$. It then follows
that $|\psi(0)|=1$, therefore $\psi$ is a constant, which is a contradiction.
Thus $M_{\psi}$ is not a $3$-isometry on $S^{2}({\mathbb{D}}).$

If $M_{\psi}$ is a $m$-isometry on $S^{2}({\mathbb{D}})$ for $m\leq2$, then
$M_{\psi}$ is also a $3$-isometry on $S^{2}({\mathbb{D}})$, which is
impossible. Hence $M_{\psi}$ is not an $m$-isometry on $S^{2}({\mathbb{D}})$
for $m\leq2$. The proof is complete.
\end{proof}

\section{Blaschke products and reducing subspaces}

The Beurling invariant subspace theorem for $M_{z}$ on the Hardy space $H^{2}
$ says an invariant subspace of $M_{z}$ on $H^{2}$ is of the form $\theta
H^{2}$ for some inner function $\theta.$ Since $M_{\theta}$ on $H^{2}$ is
still an isometry, with a finite multiplicity if $\theta$ is a finite Blaschke
product, the invariant subspace of $M_{\theta}$ is essentially known by using
Beurling-Lax-Halmos theorem. The invariant subspace of $M_{z}$ on $S_{2}%
^{2}({\mathbb{D)}}$ (equivalently on $S^{2}({\mathbb{D)}}$, $S_{1}%
^{2}({\mathbb{D)}}$, $S_{2}^{2}({\mathbb{D)}}$) was nicely described by
Korenblum in 1972 \cite{Ko72}. Furthermore, by \cite{Bo86}, if $H_{0}$ is an
invariant subspace of $M_{z}$ on $S_{2}^{2}({\mathbb{D)}}$ (in fact for
$D_{\alpha}$ with $\alpha>1$), then $H_{0}$ has the codimension one property,
i.e., the dimension of $H_{0}\ominus zH_{0}$ is one. Since $M_{\theta}$ on
$S_{1}^{2}({\mathbb{D)}}$ for a finite Blaschke product on $S_{1}%
^{2}({\mathbb{D)}}$ is a $3$-isometry, the following question seems to be natural.

\begin{problem}
What are the invariant subspaces of $M_{\theta}$ on $S_{1}^{2}({\mathbb{D)}}?$
\end{problem}

A first step would be to generalize Korenblum's result to vector-valued
version of $S_{1}^{2}({\mathbb{D)}}$. While the invariant subspaces of
$M_{\theta}$ are the same with respect to equivalent norms, the reducing
subspaces of $M_{\theta}$ could be different with respect to equivalent norms.
Starting with the paper \cite{Zhu} in 2000 where the reducing subspaces of
$M_{\theta}$ for a Blaschke product of order $2$ on the Bergman space were
characterized, there have been intensive research activities to understand the
reducing subspaces of $M_{\theta}$ on the Bergman space for a general Blaschke
product, see the recent book \cite{GH}. A recent paper \cite{Luo} discusses
the reducing subspaces of $M_{\theta}$ on the Dirichlet space.

In this section we prove that if $\psi$ is a Blaschke product of order $2,$
then $M_{\psi}$ is irreducible on $S^{2}({\mathbb{D)}}$ unless $\psi$ is
equivalent to $z^{2}$. We are unable to extend this result to other equivalent
norms since a part of our argument exploiting the fact that $M_{\psi}$ on
$S^{2}({\mathbb{D)}}$ is not a $3$-isometry, but close to being a $3$-isometry
as shown in Proposition \ref{igtisdn}.

Let $\varphi_{\alpha}$ be a automorphism of the unit disk and
$P_{\alpha}(\zeta)$ be the Poisson kernel,
\[
\varphi_{\alpha}=\frac{\alpha-z}{1-\overline{\alpha}z},P_{\alpha}(\zeta
)=\frac{1-|\alpha|^{2}}{|\zeta-\alpha|^{2}}.
\]

\begin{lemma}
\label{htipfpp} For $\alpha\in{\mathbb{D}},k\geq0$,
\[
\langle\varphi_{\alpha}^{\prime},z^{k}\varphi_{\alpha}^{\prime}\rangle_{H^{2}%
}=\frac{1+|\alpha|^{2}}{1-|\alpha|^{2}}\overline{\alpha^{k}}+k\overline
{\alpha^{k}}.
\]

\end{lemma}

\begin{proof}
Note that $\varphi_{\alpha}^{\prime}(z)=\frac{-1+|\alpha|^{2}}{(1-\overline
{\alpha}z)^{2}}=(-1+|\alpha|^{2})\sum_{n\geq0}(n+1)\overline{\alpha^{n}}z^{n}%
$. Thus
\begin{align}
\langle\varphi_{\alpha}^{\prime},z^{k}\varphi_{\alpha}^{\prime}\rangle_{H^{2}}
& =(-1+|\alpha|^{2})^{2}\left\langle \sum_{n\geq0}(n+1)\overline{\alpha^{n}%
}z^{n},z^{k}\sum_{n\geq0}(n+1)\overline{\alpha^{n}}z^{n}\right\rangle _{H^{2}%
}\nonumber\\
& =(-1+|\alpha|^{2})^{2}\sum_{l\geq0}(l+k+1)(l+1)\overline{\alpha^{l+k}}%
\alpha^{l}\nonumber\\
& =(-1+|\alpha|^{2})^{2}\sum_{l\geq0}(l+1)^{2}|\alpha|^{2l}\overline
{\alpha^{k}}+k\overline{\alpha^{k}}.\label{sum}%
\end{align}
Since
\[
\frac{1}{(1-|\alpha|^{2})^{3}}=\sum_{l\geq0}\frac{(l+1)(l+2)}{2}|\alpha|^{2l},
\]
we have
\begin{align*}
\sum_{l\geq0}(l+1)^{2}|\alpha|^{2l} = \frac{1+|\alpha|^{2}}{(1-|\alpha|^{2})^{3}}.
\end{align*}
Hence by (\ref{sum}), we get
\[
\langle\varphi_{\alpha}^{\prime},z^{k}\varphi_{\alpha}^{\prime}\rangle_{H^{2}%
}=\frac{1+|\alpha|^{2}}{1-|\alpha|^{2}}\overline{\alpha^{k}}+k\overline
{\alpha^{k}}.
\]
\end{proof}

\begin{lemma}
\label{ajfbistd} Assume $\psi=z\varphi_{\alpha},\alpha\in{\mathbb{D}}$. Then,
on $S^{2}({\mathbb{D}}),$
\[
M_{\psi}^{\ast}\psi(z)=3+\frac{1+|\alpha|^{2}}{1-|\alpha|^{2}}+\sum_{k\geq
1}\left[ (2k+2)\overline{\alpha^{k}}+\frac{1+|\alpha|^{2}}{1-|\alpha|^{2}%
}\overline{\alpha^{k}}\right] \frac{z^{k}}{k^{2}}.
\]

\end{lemma}

\begin{proof}
Notice that $(z\varphi_{\alpha}^{\prime}\overline{\varphi_{\alpha}}%
)(\zeta)=P_{\alpha}(\zeta)=\frac{1-|\alpha|^{2}}{|\zeta-\alpha|^{2}},\zeta
\in{\mathbb{T}}$, we have
\begin{align*}
\langle M_{\psi}^{\ast}\psi,1\rangle_{1} & =\langle\psi,\psi\rangle
_{1}=\langle(z\varphi_{\alpha})^{\prime},(z\varphi_{\alpha})^{\prime}%
\rangle_{H^{2}}\\
& =\langle\varphi_{\alpha}+z\varphi_{\alpha}^{\prime},\varphi_{\alpha
}+z\varphi_{\alpha}^{\prime}\rangle_{H^{2}}\\
& =1+2\int_{{\mathbb{T}}}P_{\alpha}(\zeta)\frac{|d\zeta|}{2\pi}+\langle
\varphi_{\alpha}^{\prime},\varphi_{\alpha}^{\prime}\rangle_{H^{2}}\\
& =3+\frac{1+|\alpha|^{2}}{1-|\alpha|^{2}},
\end{align*}
where we used Lemma \ref{htipfpp} in the last equality.

For $k\geq1$,
\begin{align*}
\langle M_{\psi}^{*}\psi,z^{k}\rangle_{1} & =\langle\psi,z^{k}\psi\rangle
_{1}=\langle(z\varphi_{\alpha})^{\prime}, (z^{k+1}\varphi_{\alpha})^{\prime}%
\rangle_{H^{2}}\\
& =\langle\varphi_{\alpha}+z\varphi_{\alpha}^{\prime}, (k+1)z^{k}\varphi_{\alpha
}+z^{k+1}\varphi_{\alpha}^{\prime}\rangle_{H^{2}}\\
& =0+\int_{\mathbb{T}}P_{\alpha}(\zeta)\overline{\zeta^{k}}\frac{|d\zeta
|}{2\pi}+(k+1)\int_{\mathbb{T}}P_{\alpha}(\zeta)\overline{\zeta^{k}}%
\frac{|d\zeta|}{2\pi}+\langle\varphi_{\alpha}^{\prime}, z^{k}\varphi_{\alpha
}^{\prime}\rangle_{H^{2}}\\
& =(k+2)\overline{\alpha^{k}}+\frac{1+|\alpha|^{2}}{1-|\alpha|^{2}}%
\overline{\alpha^{k}}+k\overline{\alpha^{k}}\\
& =(2k+2)\overline{\alpha^{k}}+\frac{1+|\alpha|^{2}}{1-|\alpha|^{2}}%
\overline{\alpha^{k}}.
\end{align*}
Therefore
\begin{align*}
M_{\psi}^{*}\psi(z) & =\langle M_{\psi}^{*}\psi,1\rangle_{1}+\sum_{k\geq
1}\langle M_{\psi}^{*}\psi,z^{k}\rangle_{1}\frac{z^{k}}{k^{2}}\\
& =3+\frac{1+|\alpha|^{2}}{1-|\alpha|^{2}}+\sum_{k\geq1}\left[ (2k+2)\overline
{\alpha^{k}}+\frac{1+|\alpha|^{2}}{1-|\alpha|^{2}}\overline{\alpha^{k}}\right]
\frac{z^{k}}{k^{2}}.
\end{align*}
\end{proof}

For two finite Blaschke products $\psi$ and $\phi$, $\psi$ is called to be
equivalent to $\phi$ if there exist $|a|=1,\alpha\in{\mathbb{D}}$ such that
$\phi(z)=a\varphi_{\alpha}(\psi(z))$. Note that if $\psi$ is equivalent to
$\phi$, then $M_{\psi}$ and $M_{\phi}$ have the same reducing subspaces. Now
we can prove one of the main results in this section.

\begin{theorem}
\label{rspistdti} Assume $\psi$ is a finite Blaschke product of order 2. Then
$M_{\psi}$ is reducible on $S^{2}({\mathbb{D}})$ if and only if $\psi$ is
equivalent to $z^{2}$.
\end{theorem}

\begin{proof}
If $\psi$ is equivalent to $z^{2}$, since $M_{z^{2}}$ has two minimal reducing
subspaces on $S^{2}({\mathbb{D}})$ \cite{SZ02}, we have $M_{\psi}$ is
reducible on $S^{2}({\mathbb{D}})$.

Conversely, if $M_{\psi}$ is reducible on $S^{2}({\mathbb{D}})$, we show that
$\psi$ is equivalent to $z^{2}$. Note that $\varphi_{\psi(0)}(\psi
(z))=az\varphi_{\alpha}(z)$ for some $|a|=1,\alpha\in{\mathbb{D}}$, thus
$\psi$ is equivalent to $z\varphi_{\alpha}$. So we can suppose $\psi
=z\varphi_{\alpha}$, and it is enough to show $\alpha=0$. We prove it by
contradiction. Suppose $\alpha\neq0$, and ${\mathcal{M}}$ is a reducing
subspace of $M_{\psi}$ on $S^{2}({\mathbb{D}})$. Since $\operatorname{dim}%
(S^{2}({\mathbb{D}})\ominus\psi S^{2}({\mathbb{D}}))=2$, we get
$\operatorname{dim}({\mathcal{M}}\ominus\psi{\mathcal{M}})= \operatorname{dim}%
({\mathcal{M}}^{\perp}\ominus\psi{\mathcal{M}}^{\perp})=1$, where
${\mathcal{M}}^{\perp}$ is the orthogonal complement of ${\mathcal{M}}$. Let
$f\in{\mathcal{M}}\ominus\psi{\mathcal{M}},g\in{\mathcal{M}}^{\perp}%
\ominus\psi{\mathcal{M}}^{\perp}$, then $f,g\in S^{2}({\mathbb{D}})\ominus\psi
S^{2}({\mathbb{D}})=Span\{1,K_{\alpha}\}$ where $K_{\alpha}(z)$ is the
reproducing kernel of $S^{2}({\mathbb{D}})$ at $\alpha.$ Hence $f=a+bK_{\alpha
},g=c+dK_{\alpha}$ for some $a,b,c,d\in{\mathbb{C}}$. By Proposition
\ref{igtisdn}, for $l=0,1,2,3$, we have
\begin{align*}
\Vert\psi^{n}(f+i^{l}g)\Vert_{1}^{2} & =\frac{n(n-1)}{2}\Vert\psi_{1}%
^{2}(f+i^{l}g)\Vert_{1}^{2}-n(n-2)\Vert\psi(f+i^{l}g)\Vert_{1}^{2}\\
& +\frac{(n-1)(n-2)}{2}\Vert(f+i^{l}g)\Vert_{1}^{2}-\frac{(n-1)(n-2)}%
{2}|(f+i^{l}g)(0)|^{2},
\end{align*}
here we used $\psi(0)=0$. Notice that $\langle\psi^{n}f,\psi^{k}g\rangle
_{1}=0$ for any $n,k\geq0$, therefore
\begin{align*}
0 & =\langle\psi^{n}f,\psi^{n}g\rangle_{1}=\frac{1}{4}\sum_{l=0}^{3}i^{l}%
\Vert\psi^{n}(f+i^{l}g)\Vert_{1}^{2}\\
& =\frac{n(n-1)}{2}\langle\psi^{2}f,\psi^{2}g)\rangle_{1}-n(n-2)\langle\psi
f,\psi g\rangle_{1}\\
& +\frac{(n-1)(n-2)}{2}\langle f,g\rangle_{1}-\frac{(n-1)(n-2)}{2}%
f(0)\overline{g(0)}\\
& =-\frac{(n-1)(n-2)}{2}f(0)\overline{g(0)}.
\end{align*}
So $f(0)\overline{g(0)}=0$. Without loss of generality, suppose $f(0)=0$, then
$f=a(1-K_{\alpha}),a\in{\mathbb{C}}$, thus $f_{1}=1-K_{\alpha}\in{\mathcal{M}%
}$. By
\[
0=\langle g,f\rangle_{1}=a[c+d-c-dK_{\alpha}(\alpha)],
\]
we obtain $d=0$, so $g_{1}=1\in{\mathcal{M}}^{\perp}$. Hence
\[
0=\langle\psi g_{1},\psi f_{1}\rangle_{1}=\langle\psi,\psi(1-K_{\alpha
})\rangle_{1}=M_{\psi}^{\ast}\psi(0)-M_{\psi}^{\ast}\psi(\alpha).
\]
But by Lemma \ref{ajfbistd}, we have $M_{\psi}^{\ast}\psi(0)\neq M_{\psi
}^{\ast}\psi(\alpha)$, which is a contradiction. Therefore $\alpha=0$. The
proof is complete.
\end{proof}

In the rest of this section, we discuss for a finite Blaschke product $\psi$,
when $M_{\psi}$ is unitarily equivalent to $M_{z^{n}}$.

\begin{proposition}
\label{utrotss} Let $\psi\in S^{2}({\mathbb{D}})$. If $M_{\psi}$ is unitarily
equivalent to $M_{z^{n}}$ for some $n>0$, then $\psi$ is a finite Blaschke
product with $n$ zeros.
\end{proposition}

\begin{proof}
Note that $\sigma(M_{z^{n}})=\overline{{\mathbb{D}}}$. If $M_{\psi}$ is
unitarily equivalent to $M_{z^{n}}$, then $\sigma(M_{\psi})=\psi
(\overline{{\mathbb{D}}})=\overline{{\mathbb{D}}}$ and $\sigma_e(M_{\psi})=\psi
({\mathbb{T}})= {\mathbb{T}}$, so $\psi$ is a finite Blaschke product. By $\operatorname{dim}(S^{2}({\mathbb{D}%
})\ominus\psi S^{2}({\mathbb{D}}))= \operatorname{dim}(S^{2}({\mathbb{D}%
})\ominus z^{n}S^{2}({\mathbb{D}}))=n$, we conclude that $\psi$ is a finite
Blaschke product of order $n$.
\end{proof}

\begin{lemma}
\label{scfpazk} For $k\geq0$, let $b(k)=\int_{{\mathbb{T}}}P_{\alpha}%
(\zeta)P_{-\alpha}(\zeta)\overline{\zeta^{k}}\frac{d\zeta}{2\pi}$. Then
\[
b(k)=
\begin{cases}
\frac{1-|\alpha|^{2}}{1+|\alpha|^{2}}\overline{\alpha^{2l}},\quad k=2l,\\
0,\quad k=2l+1.
\end{cases}
\]

\end{lemma}

\begin{proof}
Note that $P_{\alpha}(\zeta)P_{-\alpha}(\zeta)=\frac{(1-|\alpha|^{2})^{2}%
}{|\zeta^{2}-\alpha^{2}|^{2}}=(1-|\alpha|^{2})^{2}\left\vert \sum_{n\geq
0}\overline{\zeta^{2n}}\alpha^{2n}\right\vert ^{2}$. Thus if $k=2l+1$, then
$b(2l+1)=0$. If $k=2l$, then
\begin{align*}
b(2l) & =\int_{{\mathbb{T}}}P_{\alpha}(\zeta)P_{-\alpha}(\zeta)\overline
{\zeta^{2l}}\frac{d\zeta}{2\pi}=\int_{{\mathbb{T}}}(1-|\alpha|^{2}%
)^{2}\left\vert \sum_{n\geq0}\overline{\zeta^{2n}}\alpha^{2n}\right\vert
^{2}\overline{\zeta^{2l}}\frac{d\zeta}{2\pi}\\
& =(1-|\alpha|^{2})^{2}\sum_{n\geq0}|\alpha|^{4n}\overline{\alpha^{2l}}%
=\frac{1-|\alpha|^{2}}{1+|\alpha|^{2}}\overline{\alpha^{2l}}.
\end{align*}

\end{proof}

\begin{lemma}
\label{cfpeppns} Let $\psi=\varphi_{\alpha}\varphi_{-\alpha},\alpha
\in{\mathbb{D}}$. Then, on $S^{2}({\mathbb{D}}),$
\begin{align*}
M_{\psi}^{\ast}\psi(z) & =|\alpha|^{4}+2\frac{1+|\alpha|^{2}}{1-|\alpha|^{2}%
}+2\frac{1-|\alpha|^{2}}{1+|\alpha|^{2}}\\
& +\sum_{l\geq1}\left[ \frac{1+|\alpha|^{2}}{1-|\alpha|^{2}}2\overline
{\alpha^{2l}}+8l\overline{\alpha^{2l}}+\frac{1-|\alpha|^{2}}{1+|\alpha|^{2}%
}\overline{\alpha^{2l}}\right] \frac{z^{2l}}{4l^{2}}.
\end{align*}

\end{lemma}

\begin{proof}
By Lemma \ref{htipfpp}, we have
\begin{align*}
\langle M_{\psi}^{*}\psi,1\rangle_{1} & =\langle\psi,\psi\rangle_{1}%
=|\psi(0)|^{2}+\langle\psi^{\prime},\psi^{\prime}\rangle_{H^{2}}\\
& =|\alpha|^{4}+\langle\varphi_{\alpha}^{\prime}\varphi_{-\alpha}%
+\varphi_{\alpha}\varphi_{-\alpha}^{\prime},\varphi_{\alpha}^{\prime}%
\varphi_{-\alpha}+\varphi_{\alpha}\varphi_{-\alpha}^{\prime}\rangle_{H^{2}}\\
& =|\alpha|^{4}+\langle\varphi_{\alpha}^{\prime},\varphi_{\alpha}^{\prime
}\rangle_{H^{2}}+\langle z\varphi_{\alpha}^{\prime}\varphi_{-\alpha}%
,z\varphi_{\alpha}\varphi_{-\alpha}^{\prime}\rangle_{H^{2}}\\
& \hspace{0.5cm}+\langle z\varphi_{\alpha}\varphi_{-\alpha}^{\prime}%
,z\varphi_{\alpha}^{\prime}\varphi_{-\alpha}\rangle_{H^{2}}+\langle
\varphi_{-\alpha}^{\prime},\varphi_{-\alpha}^{\prime}\rangle_{H^{2}}\\
& =|\alpha|^{4}+2\frac{1+|\alpha|^{2}}{1-|\alpha|^{2}}+2\int_{\mathbb{T}%
}P_{\alpha}(\zeta)P_{-\alpha}(\zeta)\frac{d\zeta}{2\pi}\\
& =|\alpha|^{4}+2\frac{1+|\alpha|^{2}}{1-|\alpha|^{2}}+2\frac{1-|\alpha|^{2}%
}{1+|\alpha|^{2}},
\end{align*}
where in the last equality we used Lemma \ref{scfpazk}.

Also, for $k\geq1$,
\begin{align*}
\langle M_{\psi}^{*}\psi,z^{k}\rangle_{1} & =\langle\psi,z^{k}\psi\rangle
_{1}=\langle\psi^{\prime}, (z^{k}\psi)^{\prime}\rangle_{H^{2}}\\
& =\langle\varphi_{\alpha}^{\prime}\varphi_{-\alpha}+\varphi_{\alpha}%
\varphi_{-\alpha}^{\prime}, z^{k}\varphi_{\alpha}^{\prime}\varphi_{-\alpha}%
+z^{k}\varphi_{\alpha}\varphi_{-\alpha}^{\prime} +k z^{k-1}\varphi_{\alpha}%
\varphi_{-\alpha}\rangle_{H^{2}}\\
& =\langle\varphi_{\alpha}^{\prime}, z^{k}\varphi_{\alpha}^{\prime}\rangle_{H^{2}%
}+\langle z\varphi_{\alpha}^{\prime}\varphi_{-\alpha},z^{k+1}\varphi_{\alpha
}\varphi_{-\alpha}^{\prime}\rangle_{H^{2}}+\langle\varphi_{\alpha}^{\prime}, kz^{k-1}\varphi_{\alpha}\rangle_{H^{2}}\\
& \hspace{0.5cm}+\langle z\varphi_{\alpha}\varphi_{-\alpha}^{\prime}, z^{k+1}\varphi_{\alpha}^{\prime}\varphi_{-\alpha}\rangle_{H^{2}}+\langle
\varphi_{-\alpha}^{\prime}, z^{k}\varphi_{-\alpha}^{\prime}\rangle_{H^{2}}%
+\langle\varphi_{-\alpha}^{\prime}, kz^{k-1}\varphi_{-\alpha}\rangle_{H^{2}}\\
& =\frac{1+|\alpha|^{2}}{1-|\alpha|^{2}}\overline{\alpha^{k}}+k\overline
{\alpha^{k}}+2\int_{\mathbb{T}}P_{\alpha}(\zeta)P_{-\alpha}(\zeta
)\overline{\zeta^{k}}\frac{d\zeta}{2\pi}+k\overline{\alpha^{k}}\\
& \hspace{0.5cm}+\frac{1+|\alpha|^{2}}{1-|\alpha|^{2}}\overline{(-\alpha)^{k}%
}+k\overline{(-\alpha)^{k}}+k\overline{(-\alpha)^{k}}\\
& =\frac{1+|\alpha|^{2}}{1-|\alpha|^{2}}\left( \overline{\alpha^{k}}%
+\overline{(-\alpha)^{k}}\right) +2k\left( \overline{\alpha^{k}}%
+\overline{(-\alpha)^{k}}\right) +b(k),
\end{align*}
where $b(k)$ is the value in Lemma \ref{scfpazk}. So by Lemma \ref{scfpazk},
\[
\langle M_{\psi}^{*}\psi,z^{k}\rangle_{1}=
\begin{cases}
\frac{1+|\alpha|^{2}}{1-|\alpha|^{2}}2\overline{\alpha^{2l}}+8l\overline
{\alpha^{2l}}+\frac{1-|\alpha|^{2}}{1+|\alpha|^{2}}\overline{\alpha^{2l}%
},\quad k=2l,\\
0,\quad k=2l+1.
\end{cases}
\]
It follows that
\begin{align*}
M_{\psi}^{*}\psi(z) & =\langle M_{\psi}^{*}\psi,1\rangle_{1}+\sum_{k\geq
1}\langle M_{\psi}^{*}\psi,z^{k}\rangle_{1}\frac{z^{k}}{k^{2}}\\
& =|\alpha|^{4}+2\frac{1+|\alpha|^{2}}{1-|\alpha|^{2}}+2\frac{1-|\alpha|^{2}%
}{1+|\alpha|^{2}}\\
& \hspace{0.5cm}+\sum_{l\geq1}\left[ \frac{1+|\alpha|^{2}}{1-|\alpha|^{2}%
}2\overline{\alpha^{2l}}+8l\overline{\alpha^{2l}}+\frac{1-|\alpha|^{2}%
}{1+|\alpha|^{2}}\overline{\alpha^{2l}}\right] \frac{z^{2l}}{4l^{2}}.
\end{align*}

\end{proof}

Now we can prove the following unitarily equivalent case.

\begin{theorem}
\label{uectistd} Let $\psi\in S^{2}({\mathbb{D}})$. Then $M_{\psi}$ is
unitarily equivalent to $M_{z^{2}}$ on $S^{2}({\mathbb{D}})$ if and only if
$\psi=az^{2}$ for some $|a|=1$.
\end{theorem}

\begin{proof}
If $\psi=az^{2}$ for some $|a|=1$, then it is clear that $M_{\psi}$ is
unitarily equivalent to $M_{z^{2}}$ on $S^{2}({\mathbb{D}})$.

Conversely, if $M_{\psi}$ is unitarily equivalent to $M_{z^{2}}$ on
$S^{2}({\mathbb{D}})$, we show that $\psi=az^{2}$ for some $|a|=1$. By
Proposition \ref{utrotss}, we have $\psi$ is a finite Blaschke product of
order $2$. Since $M_{z^{2}}$ is reducible on $S^{2}({\mathbb{D}})$ with two
minimal reducing subspaces ${\mathcal{M}}_{1}=Span\{1,z^{2},z^{4},\cdots\}$
and ${\mathcal{M}}_{2}=Span\{z,z^{3},z^{5},\cdots\}$, we have $M_{\psi}$ is
reducible on $M_{z^{2}}$ on $S^{2}({\mathbb{D}})$. It then follows from
Theorem \ref{rspistdti} that $\psi$ is equivalent to $z^{2}$, i.e.
$\psi(z)=a\varphi_{\lambda}(z^{2})$ for some $|a|=1,\lambda\in{\mathbb{D}}$,
thus $M_{\psi}$ also has two minimal reducing subspaces ${\mathcal{M}}_{1}$
and ${\mathcal{M}}_{2}$.

Let $\alpha\in{\mathbb{D}}$ be such that $\alpha^{2}=\lambda$, then
$\psi(z)=-a\varphi_{\alpha}(z)\varphi_{-\alpha}(z)$, so $M_{\psi}$ is
unitarily equivalent to $M_{\varphi_{\alpha}\varphi_{-\alpha}}$. Without loss
of generality, suppose $\psi=\varphi_{\alpha}\varphi_{-\alpha}$, and there is
a unitary operator $V$ on $S^{2}({\mathbb{D}})$ such that $V^{\ast}M_{\psi
}V=M_{z^{2}}$. It is then enough to show that $\alpha=0$. We prove it by
contradiction. Suppose $\alpha\neq0$. Let $f=V1$, then $f$ is in
${\mathcal{M}}_{1}$ or ${\mathcal{M}}_{2}$. By $M_{z^{2}}^{\ast}1=0$, we get
$M_{\psi}^{\ast}f=0$, so $f=c_{1}K_{\alpha}+c_{2}K_{-\alpha},c_{1},c_{2}%
\in{\mathbb{C}}$. Note that $VM_{z^{2}}1=M_{\psi}V1=\psi f,M_{z^{2}}^{\ast
}(z^{2})=4$, we obtain that
\[
4f=VM_{z^{2}}^{\ast}(z^{2})=M_{\psi}^{\ast}V(z^{2})=M_{\psi}^{\ast}(\psi f).
\]
Thus
\[
\langle1,4f\rangle_{1}=\langle1,M_{\psi}^{*}(\psi f)\rangle_{1}=\langle
\psi,\psi f\rangle_{1}=\langle M_{\psi}^{*}\psi,f\rangle_{1}.
\]

If $f\in{\mathcal{M}}_{1}$, then $f=c_{1}(K_{\alpha}+K_{-\alpha})$, so
$\langle1,4f\rangle_{1}=8c_{1},\langle M_{\psi}^{*}\psi,f\rangle_{1}%
=c_{1}[M_{\psi}^{*}\psi(\alpha)+M_{\psi}^{*}\psi(-\alpha)]$. Hence $M_{\psi
}^{*}\psi(\alpha)+M_{\psi}^{*}\psi(-\alpha)=8$. But by Lemma \ref{cfpeppns},
we have $M_{\psi}^{*}\psi(\alpha)+M_{\psi}^{*}\psi(-\alpha)>8$, which is a contradiction.

If $f\in{\mathcal{M}}_{2}$, then $f=c_{1}(K_{\alpha}-K_{-\alpha})$. Similarly, by using $\langle z, 4f\rangle_{1} = \langle M_{\psi}^{\ast}(\psi z), f\rangle_{1}$, we also have a contradiction.

The following is the reasoning.

Since $4f = M_{\psi}^{\ast}(\psi f)$, we have $\langle z, 4f\rangle_{1} = \langle M_{\psi}^{\ast}(\psi z), f\rangle_{1}$. Note that
$\langle z, 4f\rangle_{1} = 8 \overline{c_{1}} \alpha$, and $f = c_1\sum_{k \geq 0} \frac{2}{(2k+1)^{2}} \overline{\alpha^{2k+1}}z^{2k+1}$, thus
\begin{align}\label{paetpne}
\left\langle M_{\psi}^{\ast}(\psi z), \sum_{k \geq 0} \frac{2}{(2k+1)^{2}} \overline{\alpha^{2k+1}}z^{2k+1}\right\rangle_{1} = 8 \alpha.
\end{align}
For $k \geq 0$, we have
\begin{align*}
&\langle M_{\psi}^{\ast}(\psi z), z^{2k+1}\rangle_{1}\\
& = \langle (\psi z)', (z^{2k+1}\psi)'\rangle_{H^{2}}\\
& = \langle \psi' z, (2k+1) z^{2k} \psi\rangle_{H^{2}} + \langle \psi' z,  z^{2k+1} \psi'\rangle_{H^{2}}\\
&\hspace{0.5cm} + \langle \psi, (2k+1) z^{2k} \psi\rangle_{H^{2}} + \langle \psi, z^{2k+1} \psi'\rangle_{H^{2}}\\
& = 2(2k+1) \overline{\alpha^{2k}}+\langle \psi', z^{2k}\psi'\rangle_{H^2} + \langle 1, (2k+1) z^{2k}\rangle_{H^2} + 2 \overline{\alpha^{2k}} \quad \text{by Lemmas 4.2, 4,6}\\
& = 2(2k+2) \overline{\alpha^{2k}} + 2 \frac{1+|\alpha|^{2}}{1-|\alpha|^{2}} \overline{\alpha^{2k}} + 2 \frac{1-|\alpha|^{2}}{1+|\alpha|^{2}} \overline{\alpha^{2k}}  \\
&\hspace{0.5cm} + 2k \overline{\alpha^{2k}}  + \langle 1, (2k+1) z^{2k}\rangle_{H^2}.
\end{align*}
It is then clear that (\ref{paetpne}) can not hold.

Therefore $\alpha=0$. The proof is complete.
\end{proof}

\section{Complete Nevanlinna-Pick kernel and some applications}

The general Nevalinna-Pick interpolation problem tries to answer the following question.

\begin{problem}
\label{problem}Let $H$ be a Hilbert function space on a set $X$ with the
reproducing kernel $K_{w}(z),$ let $\lambda_{1},\cdots,\lambda_{N}$ be points
of $X,$ and let $w_{1},\cdots,w_{N}$ be complex numbers. When does there exist
a multiplier $\phi$ of $H$ of norm at most one that interpolates each
$\lambda_{i}$ to $w_{i}?$
\end{problem}

Since $M_{\phi}^{\ast}K_{\lambda_{i}}(z)=\overline{\phi(\lambda_{i}%
)}K_{\lambda_{i}}(z)=\overline{w_{i}}K_{\lambda_{i}}(z),$ a necessary
condition is that the $N\times N$ Pick matrix
\begin{equation}
\left[ (1-\overline{w_{i}}w_{j})K_{\lambda_{i}}(\lambda_{j})\right]
\label{pick}%
\end{equation}
is positive semi-definite. On Hardy space $H^{2}({\mathbb{D)}}$, the above
condition is also sufficient, this is the solution of classical Nevalinna-Pick
interpolation problem. Through the work of Agler, McCarthy, McCullough,
Quiggin and others, the theory of complete Pick kernel emerged, we refer to
\cite{AM} for details.

\begin{definition}
A kernel $K_{w}(z)$ on a set $X$ has the scalar Pick property if the positive
semi-definite condition of Pick matrix as in (\ref{pick}) is also sufficient
for Problem \ref{problem}. The kernel $K_{w}(z)$ has the $s\times t$ matrix
Pick property if the the positive semi-definite condition of Pick matrix for
the matrix-valued Nevalinna-Pick interpolation problem is also sufficient. The
kernel $K_{w}(z)$ has the complete Pick property if it has the $s\times t$
matrix Pick property for all positive integers $s$ and $t.$
\end{definition}

The Szeg\H{o} kernel of the Hardy space $H^{2}({\mathbb{D)}}$ has the complete
Pick property. The kernel for the Dirichlet space $D^{2}({\mathbb{D)}}$ is
also a complete Nevanlinna-Pick kernel \cite{Ag1}. It turns out the complete
Pick property can be characterized nicely while the scalar Pick property or
the matrix Pick property is more difficulty to study. We recall Theorem 7.33
and Lemma 7.38 (Kaluza's Lemma) \cite{AM}.

\begin{theorem}
\label{completep}\cite{AM} Suppose $H$ is a holomorphic Hilbert space on
${\mathbb{D}}$ with kernel function
\begin{equation}
K_{w}(z)={\sum_{n\geq0}}a_{n}(\overline{w}z)^{n}{\text{ where }}%
a_{n}=1/\left\Vert z^{n}\right\Vert ^{2}.\label{cpick}%
\end{equation}
Let the Taylor coefficients of $1/K_{w}(z)$ at zero be given by
\[
\frac{1}{{\sum_{n\geq0}}a_{n}t^{n}}={\sum_{n\geq0}}c_{n}t^{n}.
\]
Then $H$ has the complete Pick property if and only if
\[
c_{n}\leq0{\text{ for all }}n\geq1.
\]
In particular, if $a_{0}=1$ and $a_{n}$ satisfies
\[
a_{n}^{2}\leq a_{n-1}a_{n+1}{\text{ for all }}n\geq1,
\]
Then $H$ has the complete Pick property.
\end{theorem}

\begin{theorem}
The space $S_{1}^{2}({\mathbb{D)}}$ has the complete Pick property.
\end{theorem}

\begin{proof}
The proof is by verifying that for all $n\geq1,$
\[
a_{n}=\frac{2^{2}}{\left( n+1\right) ^{2}(n+2)^{2}}<\frac{4}{n(n+1)(n+2)(n+3)}%
=a_{n-1}a_{n+1}.
\]

\end{proof}

We note that $S^{2}({\mathbb{D)}}$ or $S_{2}^{2}({\mathbb{D)}}$ does not has
the complete Pick property. This again follows from Theorem \ref{completep} by
verifying that
\begin{align*}
\frac{1}{1+{\sum_{n\geq1}}\left(  t^{n}/n^{2}\right)  } &  =1-t+\frac{3}%
{4}t^{2}+\cdots,\\
\frac{1}{{\sum_{n\geq0}}t^{n}/(1+n^{2})} &  =1-\frac{1}{2}t+\frac{1}{20}%
t^{2}+\cdots.
\end{align*}
In fact, $S^{2}({\mathbb{D)}}$ or $S_{2}^{2}({\mathbb{D)}}$ does not has the
scalar Pick property. The following example is adapted from Problem 3 of
Exercises 5.3 in \cite{EM}, where a similar situation occurs for another
definition of Dirichlet space different from $D^{2}({\mathbb{D}})$ as in
(\ref{defdiri}).

\begin{proposition}
The space $S^{2}({\mathbb{D)}}$ or $S_{2}^{2}({\mathbb{D)}}$ does not have the
scalar Pick property.
\end{proposition}

\begin{proof}
The reproducing kernel of $S^{2}({\mathbb{D)}}$ is
\[
K_{\lambda}(z)=1+{\sum_{n\geq1}}\frac{(\overline{\lambda}z)^{n}}{n^{2}}.
\]
If it has the scalar Pick property, then the necessary condition for the
existence of $\varphi(z)={\sum_{n\geq0}}\varphi_{n}z^{n}$ such that
\[
\varphi(0)=0,\varphi(z_{0})=w_{0}{\text{ and }}\left\Vert M_{\varphi
}\right\Vert \leq1.
\]
is the Pick matrix is positive semi-definite, i.e.,
\begin{equation}
\left[
\begin{array}
[c]{cc}%
1 & 1\\
1 & (1-\left\vert w_{0}\right\vert ^{2})K_{z_{0}}(z_{0})
\end{array}
\right]  \geq0.\label{one}%
\end{equation}
Equivalently
\[
(1-\left\vert w_{0}\right\vert ^{2})\left[  1+{\sum_{n\geq1}}\frac{\left\vert
z_{0}\right\vert ^{2n}}{n^{2}}\right]  \geq1.
\]
On the other hand, $\left\Vert M_{\varphi}\right\Vert \leq1$ implies that
$\left\Vert M_{\varphi}z\right\Vert \leq\left\Vert M_{\varphi}\right\Vert
\left\Vert z\right\Vert \leq1.$ That is,
\[
\left\Vert M_{\varphi}z\right\Vert ^{2}=\left\Vert {\sum_{n\geq0}}\varphi
_{n}z^{n+1}\right\Vert ^{2}={\sum_{n\geq1}}(n+1)^{2}\left\vert \varphi
_{n}\right\vert ^{2}\leq1,
\]
since $\varphi_{0}=\varphi(0)=0.$ Hence
\begin{align}
\left\vert w_{0}\right\vert ^{2} &  =\left\vert \varphi(z_{0})\right\vert
^{2}=\left\vert {\sum_{n\geq1}}\varphi_{n}z_{0}^{n}\right\vert ^{2}\leq\left(
{\sum_{n\geq1}}\left\vert \varphi_{n}\right\vert \left\vert z_{0}\right\vert
^{n}\right)  ^{2}\label{two}\\
&  =\left(  {\sum_{n\geq1}}(n+1)\left\vert \varphi_{n}\right\vert
\frac{\left\vert z_{0}\right\vert ^{n}}{n+1}\right)  ^{2}\leq{\sum_{n\geq1}%
}\frac{\left\vert z_{0}\right\vert ^{2n}}{(n+1)^{2}}.\nonumber
\end{align}
Let $z_{0}=0.5$ and $w_{0}^{2}=0.1,$ then (\ref{one}) holds since
\[
(1-0.1)\ast\left[  1+{\sum_{n=1}^{\infty}}\frac{\left(  0.5\right)  ^{2n}%
}{n^{2}}\right]  \approx0.9\ast1.2677\approx1.1409>1.
\]
But (\ref{two}) does not holds since
\[
{\sum_{n=1}^{\infty}}\frac{\left(  0.5\right)  ^{2n}}{(n+1)^{2}}%
\approx0.0706\ngeqslant0.1.
\]
The proof for $S_{2}^{2}({\mathbb{D)}}$ is similar.
\end{proof}

There are several applications of having the complete Pick property. We focus
on two applications related to the multiplication operators and composition
operators. The first is the following (scalar) Toeplitz-Corona theorem, which
is a straightforward corollary of the much more general Theorem 8.57
\cite{AM}. But there are hidden multiplication operators on vector-valued
$S_{1}^{2}({\mathbb{D)}}$ space in our statements.

\begin{theorem}
Let $\varphi_{1},\cdots,\varphi_{N}$ be in $S_{1}^{2}({\mathbb{D)}}$. The
following are equivalent.\newline(1) There exist functions $\psi_{1}%
,\cdots,\psi_{N}$ in $S_{1}^{2}({\mathbb{D)}}$ such that
\[
{\sum_{i=1}^{N}}\varphi_{i}\psi_{i}=1
\]
and
\[
{\sum_{i=1}^{N}}\left\Vert M_{\psi_{i}}h\right\Vert ^{2}\leq\frac{1}%
{\delta^{2}}\left\Vert h\right\Vert ^{2}{\text{ for }}h\in S_{1}%
^{2}({\mathbb{D)}}{\text{.}}%
\]
\newline(2) The multipliers $M_{\varphi_{i}}$ on $S_{1}^{2}({\mathbb{D)}}$
satisfy the inequality
\[
{\sum_{i=1}^{N}}M_{\varphi_{i}}^{\ast}M_{\varphi_{i}}\geq\delta^{2}I.
\]
\newline(3) The function
\[
\left[  {\sum_{i=1}^{N}}\overline{\varphi_{i}(w)}\varphi_{i}(z)-\delta
^{2}\right]  K_{w}(z)
\]
is positive semi-definite on ${\mathbb{D}}{\mathbb{\times}}{\mathbb{D}}$.
\end{theorem}

The following result follows from Theorem \ref{norm}.

\begin{corollary}
Let $\varphi_{1},\cdots,\varphi_{N}$ be in $S_{1}^{2}({\mathbb{D)}}$. If
\begin{equation}
\left[  {\sum_{i=1}^{N}}\overline{\varphi_{i}(w)}\varphi_{i}(z)-\delta
^{2}\right]  K_{w}(z)\label{ckernel}%
\end{equation}
is positive semi-definite, then there exist functions $\psi_{1},\cdots
,\psi_{N}$ in $S_{1}^{2}({\mathbb{D)}}$ such that
\begin{equation}
{\sum_{i=1}^{N}}\varphi_{i}\psi_{i}=1{\text{ and }\sum_{i=1}^{N}}\left\Vert
\psi_{i}\right\Vert ^{2}\leq\frac{1}{\delta^{2}}.\label{corona}%
\end{equation}
Conversely, if (\ref{corona}) holds, then the function as in (\ref{ckernel})
with $\delta$ replaced by $\delta/2\sqrt{2}$ is positive semi-definite.
\end{corollary}

\begin{proof}
If (\ref{ckernel}) is positive semi-definite, then by Toeplitz-Corona
Theorem,
\[
{\sum_{i=1}^{N}}\left\Vert \psi_{i}\right\Vert ^{2}={\sum_{i=1}^{N}}\left\Vert
M_{\psi_{i}}1\right\Vert ^{2}\leq\frac{1}{\delta^{2}}.
\]
Conversely, if (\ref{corona}) holds, then by Theorem \ref{norm}, for $h\in
S_{1}^{2}({\mathbb{D)}}$,
\[
{\sum_{i=1}^{N}}\left\Vert M_{\psi_{i}}h\right\Vert ^{2}={\sum_{i=1}^{N}%
}\left\Vert \psi_{i}h\right\Vert ^{2}\leq{\sum_{i=1}^{N}}8\left\Vert \psi
_{i}\right\Vert ^{2}\left\Vert h\right\Vert ^{2}\leq\frac{8}{\delta^{2}%
}\left\Vert h\right\Vert ^{2}.
\]
Again by Toeplitz-Corona Theorem, the function as in (\ref{ckernel}) with
$\delta$ replaced by $\delta/2\sqrt{2}$ is positive semi-definite.
\end{proof}

The second application of the complete Pick kernel of $S_{1}^{2}({\mathbb{D)}%
}$ is inspired by Theorem 1 in \cite{Jury}, which derives an upper bound for
the norm of a weighted composition operator and provides a new insight into a
known upper bound for the norm of a composition operator on the Hardy space
and weighted Bergman spaces, which says on $H^{2}({\mathbb{D)}}$, for a
self-map $\varphi$ of the disk,
\begin{equation}
\left\Vert C_{\varphi}\right\Vert ^{2}\leq\frac{1+\left\vert \varphi
(0)\right\vert }{1-\left\vert \varphi(0)\right\vert }.\label{upperb}%
\end{equation}
This upper bound is usually proved by Littlewood subordination principle. This
upper bound also says on $H^{2}({\mathbb{D)}}$, $C_{\varphi}$ is automatically
bounded. In general $C_{\varphi}$ is not necessarily bounded on $S_{1}%
^{2}({\mathbb{D)}}$ for a self-map $\varphi$ of the disk, the boundedness
condition is studied in \cite{Barbara} \cite{CH}. One may ask, if $C_{\varphi
}$ is bounded, does a similar inequality holds? The following example shows
this is impossible.

\begin{example}
\label{zknorm}Note that $C_{z^{k}}$ is a diagonal operator, so we can compute
its norm as follows:
\[
\left\Vert C_{z^{k}}f\right\Vert ^{2}=\left\Vert {\sum_{n\geq0}}f_{n}%
z^{kn}\right\Vert ^{2}={\sum_{n\geq0}}\frac{(kn+1)(kn+2)}{2}\left\vert
f_{n}\right\vert ^{2}.
\]
Hence
\[
\left\Vert C_{z^{k}}\right\Vert ^{2}=\sup_{n\geq0}\frac{(kn+1)(kn+2)}{2}%
\div\frac{(n+1)(n+2)}{2}=k^{2}.
\]
That is $\left\Vert C_{z^{k}}\right\Vert =k,$ and $\left\Vert C_{\varphi
}\right\Vert \leq C\frac{1+\left\vert \varphi(0)\right\vert }{1-\left\vert
\varphi(0)\right\vert }$ can not hold for any constant $C.$
\end{example}

Since the multipliers of $H^{2}({\mathbb{D)}}$ is $H^{\infty}({\mathbb{D)}}$
and $\varphi$ is a self-map of the disk if and only if $\left\Vert M_{\varphi
}\right\Vert \leq1$ on $H^{2}({\mathbb{D)}}$, if we require $\left\Vert
M_{\varphi}\right\Vert \leq1$ on $S_{1}^{2}({\mathbb{D)}}$, then we do get a
similar upper bound as in (\ref{upperb}) for $\left\Vert C_{\varphi
}\right\Vert .$ In fact our proof is valid for a more general reproducing
kernel, which even yields a new upper bound for the norm of a composition
operator on the Dirichlet space. We start the proof similarly as in the proof
of Theorem 1 \cite{Jury} with appropriate modifications, then quickly run into
difficulty because the particular form of Szeg\H{o} kernel makes the Pick
kernel has a nice form $(1-\overline{\varphi(w)}\varphi(z))/(1-\overline
{w}z),$ so a couple of new ideas are needed.

Let $H_{K}$ be a holomorphic Hilbert space on ${\mathbb{D}}$ with reproducing kernel $K_{w}(z).$ Let $\varphi\in H_{K}$ be such that $\left\Vert
M_{\varphi}\right\Vert \leq1$ on $H_{K}$, which is equivalent to the kernel
\[
K_{w}^{\varphi}(z)=(1-\overline{\varphi(w)}\varphi(z))K_{w}(z)\geq0.
\]
The assumption $\left\Vert M_{\varphi}\right\Vert \leq1$ implies that
$\varphi$ is a self-map of the disk. Let $H_{K}(\varphi)$ be the (generalized)
de Branges-Rovnyak space with reproducing kernel $K_{w}^{\varphi}(z).$ For any
$f\in H({\mathbb{D}}),$ there is a densely defined operator $M_{f}^{\ast},$
defined on the kernel by $M_{f}^{\ast}K_{w}(z)=\overline{f(w)}K_{w}(z)$ and
extended linearly. Similarly $C_{\varphi}^{\ast}$ is densely defined on the
kernel by $C_{\varphi}^{\ast}K_{w}(z)=K_{\varphi(w)}(z).$ Here $M_{f}^{\ast}$
and $C_{\varphi}^{\ast}$ are only formal adjoints of $M_{f}$ and $C_{\varphi
},$ and when $M_{f}^{\ast}$ and $C_{\varphi}^{\ast}$ are bounded, they become
the actual adjoints.

\begin{theorem}
Assume $K_{w}(z)$ is of the form
\[
K_{w}(z)={\sum_{n\geq0}}a_{n}(\overline{w}z)^{n}{\text{ where }}%
a_{n}=1/\left\Vert z^{n}\right\Vert ^{2}\leq1{\text{ for all }}n\geq0,
\]
and $\varphi \in H_{K}$ satisfies $\left\Vert M_{\varphi}\right\Vert \leq1$ on $H_{K}$. Then for any $f\in H_{K}(\varphi),$ the operator $C_{\varphi}^{\ast}M_{f}^{\ast}$
is bounded on $H_{K}$ and $\left\Vert C_{\varphi}^{\ast}M_{f}^{\ast
}\right\Vert \leq\left\Vert f\right\Vert _{H_{K}(\varphi)}.$
\end{theorem}

\begin{proof}
By rescaling, we can assume $\left\Vert f\right\Vert _{H_{K}(\varphi)}=1.$ We
need to prove $\left\Vert C_{\varphi}^{\ast}M_{f}^{\ast}\right\Vert \leq1.$
Put $f_{0}=f$ and chose unit vectors $f_{1},f_{2},\cdots$ such that $\left\{
f_{n},n\geq0\right\}  $ is an orthonormal basis of $H_{K}(\varphi).$ Then
\[
K_{w}^{\varphi}(z)=(1-\overline{\varphi(w)}\varphi(z))K_{w}(z)={\sum_{n\geq0}%
}\overline{f_{n}(w)}f_{n}(z),
\]
and the kernel $K_{1}$
\[
K_{1}:=K_{w}(z)-\frac{\overline{f(w)}f(z)}{1-\overline{\varphi(w)}\varphi
(z)}={\sum_{n\geq1}}\frac{\overline{f_{n}(w)}f_{n}(z)}{1-\overline{\varphi
(w)}\varphi(z)}%
\]
is therefore positive semi-definite being the sum of positive definite
semi-definite kernels.

We need to prove $\left\Vert C_{\varphi}^{\ast}M_{f}^{\ast}\right\Vert \leq1.$
For any $N$ distinct points $w_{1},\cdots,w_{N}$ in ${\mathbb{D}}$ and complex
numbers $c_{1},\cdots,c_{N},$ let
\[
h={\sum_{i=1}^{N}}c_{i}K_{w_{i}}(z),
\]
Then $\left\langle C_{\varphi}^{\ast}M_{f}^{\ast}h,C_{\varphi}^{\ast}%
M_{f}^{\ast}h\right\rangle \leq\left\langle h,h\right\rangle $ is the same as
\begin{align*}
&  \left\langle h,h\right\rangle -\left\langle C_{\varphi}^{\ast}M_{f}^{\ast
}h,C_{\varphi}^{\ast}M_{f}^{\ast}h\right\rangle \\
&  ={\sum_{i=1}^{N}}c_{i}\overline{c_{j}}K_{w_{i}}(w_{j})-{\sum_{i=1}^{N}%
}c_{i}\overline{c_{j}}\overline{f(w_{i})}f(w_{j})K_{\varphi(w_{i})}%
(\varphi(w_{j}))\geq0.
\end{align*}
In other word, we need to prove the kernel $K_{2}$
\[
K_{2}:=K_{w}(z)-\overline{f(w)}f(z)K_{\varphi(w)}(\varphi(z))
\]
is positive semi-definite. Note that
\begin{align*}
K_{2}-K_{1} &  =\frac{\overline{f(w)}f(z)}{1-\overline{\varphi(w)}\varphi
(z)}-\overline{f(w)}f(z)K_{\varphi(w)}(\varphi(z))\\
&  =\overline{f(w)}f(z)\left[  \frac{1}{1-\overline{\varphi(w)}\varphi
(z)}-K_{\varphi(w)}(\varphi(z))\right] \\
&  =\overline{f(w)}f(z)\left[  {\sum_{n\geq0}}\left[  \overline{\varphi
(w)}\varphi(z)\right]  ^{n}-{\sum_{n\geq0}}a_{n}\left[  \overline{\varphi
(w)}\varphi(z)\right]  ^{n}\right] \\
&  ={\sum_{n\geq0}}(1-a_{n})\overline{f(w)}f(z)\left[  \overline{\varphi
(w)}\varphi(z)\right]  ^{n}%
\end{align*}
is positive semi-definite being the sum of positive semi-definite kernels.
\end{proof}

\begin{corollary}
\label{normestimate}Let $K_{w}(z)$ be the same as in previous theorem. If
$\left\Vert M_{\varphi}\right\Vert \leq1$ on $H_{K},$ then the operator
$C_{\varphi}$ is bounded on $H_{K}$ and
\[
\left\Vert C_{\varphi}\right\Vert ^{2}\leq\frac{1+\left\vert \varphi
(0)\right\vert }{1-\left\vert \varphi(0)\right\vert }.
\]

\end{corollary}

\begin{proof}
Let
\[
f(z)=1-\overline{\varphi(0)}\varphi(z).
\]
Then $f(z)=K_{0}^{\varphi}(z)$ and $\left\Vert f\right\Vert _{H_{K}(\varphi
)}^{2}=\left\langle K_{0}^{\varphi}(z),K_{0}^{\varphi}(z)\right\rangle
=1-\left\vert \varphi(0)\right\vert ^{2}.$ Note also
\[
\left\Vert M_{f}\right\Vert =\left\Vert I-\overline{\varphi(0)}M_{\varphi
}\right\Vert \geq1-\left\vert \varphi(0)\right\vert \left\Vert M_{\varphi
}\right\Vert \geq1-\left\vert \varphi(0)\right\vert {\text{ and }}\left\Vert
M_{1/f}\right\Vert \leq1/(1-\left\vert \varphi(0)\right\vert ).
\]
Hence
\begin{align*}
\left\Vert C_{\varphi}^{\ast}\right\Vert ^{2} & =\left\Vert C_{\varphi}^{\ast
}M_{f}^{\ast}M_{1/f}^{\ast}\right\Vert ^{2}\leq\left\Vert C_{\varphi}^{\ast
}M_{f}^{\ast}\right\Vert ^{2}\left\Vert M_{1/f}^{\ast}\right\Vert ^{2}\\
& \leq\left( 1-\left\vert \varphi(0)\right\vert ^{2}\right) /(1-\left\vert
\varphi(0)\right\vert )^{2}=\left( 1+\left\vert \varphi(0)\right\vert \right)
/(1-\left\vert \varphi(0)\right\vert ).
\end{align*}
The proof is complete.
\end{proof}

\begin{corollary}
Let $\varphi\in D^{2}({\mathbb{D)}}$ be such that $\left\Vert M_{\varphi
}\right\Vert \leq1.$ Then $C_{\varphi}$ bounded on $D^{2}({\mathbb{D)}}$, and
\[
\frac{1}{\left\vert \varphi(0)\right\vert ^{2}}\ln\frac{1}{1-\left\vert
\varphi(0)\right\vert ^{2}}\leq\left\Vert C_{\varphi}\right\Vert ^{2}\leq
\frac{1+\left\vert \varphi(0)\right\vert }{1-\left\vert \varphi(0)\right\vert
}%
\]
where $\varphi(0)\neq0,$ and $\left\Vert C_{\varphi}\right\Vert =1$ when
$\varphi(0)=0.$
\end{corollary}

\begin{proof}
The lower bound follows from $\left\Vert C_{\varphi}^{\ast}\right\Vert
^{2}\geq\left\Vert C_{\varphi}^{\ast}1\right\Vert ^{2}=\left\Vert C_{\varphi
}^{\ast}K_{0}\right\Vert ^{2}=\left\Vert K_{\varphi(0)}\right\Vert ^{2},$
where $K_{\varphi(0)}$ is the reproducing kernel of $D^{2}({\mathbb{D)}}$ as
in (\ref{kernel}).
\end{proof}

The condition $\left\Vert M_{\varphi}\right\Vert \leq1$ is in general only a
sufficient condition for $C_{\varphi}$ to be bounded as Example \ref{zknorm}
shows. In fact, if $\left\Vert M_{\varphi}\right\Vert \leq1,$ then
$C_{\varphi}$ on $S_{1}^{2}({\mathbb{D)}}$ is a Hilbert-Schmidt operator.
Hilbert spaces of holomorphic functions on the unit disk, such as $S_{1}%
^{2}({\mathbb{D)}}$, whose functions extend continuously to the closure of the
unit disk are referred to as boundary-regular spaces. Then, it follows from
Theorem 2.1 in \cite{Shapiro} that if $\varphi$ is a self-map of the disk
which lies in $S_{1}^{2}({\mathbb{D)}}$, then $C_{\varphi}$ is compact if and
only if $\left\Vert \varphi\right\Vert _{\infty}<1.$ We observe here that, as
in the Hardy space, the Bergman space and the Dirichlet space, if $\left\Vert
\varphi\right\Vert _{\infty}<1,$ then $C_{\varphi}$ on $S_{1}^{2}%
({\mathbb{D)}}$ is a Hilbert-Schmidt operator.

\begin{proposition}
\label{schmidt}Let $\varphi\in S_{1}^{2}({\mathbb{D)}}$ be such that
$\left\Vert \varphi\right\Vert _{\infty}<1$ or $\left\Vert M_{\varphi
}\right\Vert \leq1.$ Then $C_{\varphi}$ is a Hilbert-Schmidt operator on
$S_{1}^{2}({\mathbb{D)}}$.
\end{proposition}

\begin{proof}
Assume $\left\Vert \varphi\right\Vert _{\infty}<1.$ Note that for $n\geq1, $
\begin{align*}
\left\Vert \varphi^{n}\right\Vert _{S_{1}^{2}}^{2} &  =\left\Vert \varphi
^{n}\right\Vert _{H^{2}}^{2}+\frac{3}{2}\left\Vert n\varphi^{n-1}%
\varphi^{\prime}\right\Vert _{A^{2}}^{2}+\frac{1}{2}\left\Vert n\varphi
^{n-1}\varphi^{\prime}\right\Vert _{H^{2}}^{2}\\
&  \leq n^{2}\left\Vert \varphi\right\Vert _{\infty}^{2(n-1)}\left(
\left\Vert \varphi\right\Vert _{H^{2}}^{2}+\frac{3}{2}\left\Vert
\varphi^{\prime}\right\Vert _{A^{2}}^{2}+\frac{1}{2}\left\Vert \varphi
^{\prime}\right\Vert _{H^{2}}^{2}\right)  =n^{2}\left\Vert \varphi\right\Vert
_{\infty}^{2(n-1)}\left\Vert \varphi\right\Vert _{S_{1}^{2}}^{2}.
\end{align*}
Therefore
\begin{align*}
{\sum_{n\geq0}}\left\Vert C_{\varphi}\frac{z^{n}}{\left\Vert z^{n}\right\Vert
}\right\Vert ^{2} &  =1+{\sum_{n\geq1}}\frac{\left\Vert \varphi^{n}\right\Vert
^{2}}{\left\Vert z^{n}\right\Vert ^{2}}\leq1+{\sum_{n\geq1}}\frac
{2n^{2}\left\Vert \varphi\right\Vert _{S_{1}^{2}}^{2}}{(n+1)(n+2)}\left\Vert
\varphi\right\Vert _{\infty}^{2(n-1)}\\
&  \leq1+2\left\Vert \varphi\right\Vert _{S_{1}^{2}}^{2}/\left(  1-\left\Vert
\varphi\right\Vert _{\infty}^{2}\right)  .
\end{align*}
In the case $\left\Vert M_{\varphi}\right\Vert \leq1,$ the result follows by
noting that $\left\Vert \varphi^{n}\right\Vert ^{2}=\left\Vert M_{\varphi^{n}%
}1\right\Vert ^{2}\leq\left\Vert M_{\varphi}\right\Vert ^{2n}\leq1$.
\end{proof}

Now we can show that for $\varphi\in S_{1}^{2}({\mathbb{D}})$, $\|\varphi
\|_{\infty}=\|M_{f}\|$ if and only if $\varphi$ is a constant. Recall that for
a bounded operator $T$ on a Hilbert space $H$, $T$ is called hyponormal if
$T^{*}T-TT^{*}\geq0$. If $T$ is hyponormal, then the norm of $T$ equals the
spectral radius of $T$, i.e. $\|T\|=r(T)$ (\cite{C91}).

\begin{corollary}
Assume $\varphi\in S_{1}^{2}({\mathbb{D}})$ is not a constant. Then
$\|\varphi\|_{\infty}<\|M_{\varphi}\|$. Consequently, $M_{\varphi}$ is hyponormal on
$S_{1}^{2}({\mathbb{D}})$ if and only if $\varphi$ is a constant.
\end{corollary}

\begin{proof}
Let $\psi=\frac{\varphi}{\|M_{\varphi}\|}$, then $\|M_{\psi}\|=1$. Thus by
Proposition \ref{schmidt}, $C_{\psi}$ is a Hilbert-Schmidt operator. It then
follows from \cite[Theorem 2.1]{Shapiro} that $\|\psi\|_{\infty}<1$. So
$\|\varphi\|_{\infty}<\|M_{\varphi}\|$. Recall from (\ref{spectrum}) that
$\sigma(M_{\varphi})=\varphi(\overline{{\mathbb{D}}})$, hence $r(M_{\varphi
})=\|\varphi\|_{\infty}$. If $M_{\varphi}$ is hyponormal, then $\|M_{\varphi
}\|=r(M_{\varphi})=\|\varphi\|_{\infty}$, therefore $\varphi$ is a constant.
If $\varphi$ is a constant, we have $M_{\varphi}$ is hyponormal.
\end{proof}

\begin{corollary}
Let $\varphi\in S_{1}^{2}({\mathbb{D)}}$ be such that $\left\Vert
\varphi\right\Vert _{S_{1}^{2}}\leq1/(2\sqrt{2}).$ Then $C_{\varphi}$ is a
Hilbert-Schmidt operator on $S_{1}^{2}({\mathbb{D)}}$ and
\[
\left\Vert C_{\varphi}\right\Vert ^{2}\leq\frac{1+\left\vert \varphi
(0)\right\vert }{1-\left\vert \varphi(0)\right\vert }.
\]

\end{corollary}

\begin{proof}
Assume $\left\Vert \varphi\right\Vert _{S_{1}^{2}}\leq1/(2\sqrt{2}).$ By
Theorem \ref{norm}, $\left\Vert M_{\varphi}\right\Vert \leq1.$ So $C_{\varphi
}$ is a Hilbert-Schmidt operator. The upper bound for $\left\Vert C_{\varphi
}\right\Vert $ follows from Corollary \ref{normestimate}.
\end{proof}

\section{Higher order derivative Hardy spaces}

Using high order derivatives, for a positive integer $m,$ we may define the
following spaces,
\[
K_{m}^{2}({\mathbb{D)}}{\mathbb{=}}\left\{  f\in H({\mathbb{D)}}:\left\Vert
f\right\Vert _{K_{m}^{2}}^{2}={\sum_{n\geq0}}\frac{\left(  n+1\right)
\cdots(n+m+1)}{(m+1)!}\left\vert f_{n}\right\vert ^{2}<\infty\right\}  .
\]
Then $K_{m}^{2}({\mathbb{D)}}$ is norm equivalent to $D_{\alpha}$ with
$\alpha=m+1.$ These spaces are related to the Hardy space and the Bergman
space in the following way: for $m=2k-1$, $K_{m}^{2}({\mathbb{D)}}$ is norm
equivalent to the space
\[
\left\{  f\in H({\mathbb{D)}}:\left\Vert f\right\Vert ^{2}=|f(0)|^{2}%
+\left\Vert f^{(k)}\right\Vert _{H^{2}}^{2}<\infty\right\}  ,
\]
where $f^{(k)}(z)$ is the $k$-th derivative of $f(z),$ for $m=2k-2$,
$K_{m}^{2}({\mathbb{D)}}$ is norm equivalent to the space
\[
\left\{  f\in H({\mathbb{D)}}:\left\Vert f\right\Vert ^{2}=|f(0)|^{2}%
+\left\Vert f^{(k)}\right\Vert _{A^{2}}^{2}<\infty\right\}  .
\]
Again the reproducing kernel of $K_{m}^{2}({\mathbb{D)}}$ has a closed form by
using logarithmic functions, is a complete Nevalinna-Pick kernel, and $M_{z}$
on $K_{m}^{2}({\mathbb{D)}}$ is an $(m+2)$-isometry. In fact the following
theorem holds.

\begin{theorem}
Assume $\psi$ is not a constant. Then $M_{\psi}$ on $K_{m}^{2}({\mathbb{D)}}$
is a $k$-isometry for some $k\geq1$ if and only if $k=m+2$ and $\psi$ is a
finite Blaschke product.
\end{theorem}

Furthermore, Corollary \ref{normestimate} is also valid on $K_{m}%
^{2}({\mathbb{D)}}$.

\bigskip\bigskip

Caixing Gu

Department of Mathematics, California Polytechnic State University,

San Luis O-bispo, CA 93407, USA

E-mail address: cgu@calpoly.edu \bigskip\bigskip

Shuaibing Luo

College of Mathematics and Econometrics, Hunan University,

Changsha, Hunan, 410082, PR China

E-mail address: shuailuo2@126.com \bigskip\bigskip

Mathematics Subject Classification (2010). Primary 47B32, 47B33, 47B35,
Secondary 30J10, 30H10, 30H50.

Keywords. Derivative Hardy space, composition operator, multipication
operator, $m$-isometry, reducing subspace, Nevalinna-Pick kernel.

\end{document}